\documentclass[a4paper, 10pt]{article}
\usepackage[utf8]{inputenc}
\usepackage[left=20 mm, bottom=20 mm, right=20 mm, top=20 mm]{geometry}
\usepackage{amsmath, amssymb, amsthm, csquotes, tikz, float}
\usepackage{hyperref}
\usepackage[english]{babel}
\usepackage{graphicx,epstopdf,epsfig}
\usepackage{amsfonts,epsfig,fancyhdr, graphics, hyperref,amsmath, tikz,amssymb}
\usepackage{geometry}
\usepackage{tikz}
\usetikzlibrary{decorations.markings}
\newtheorem{theorem}{Theorem}[section]
\newtheorem{lemma}{Lemma}[section]

\newtheorem{defn}{Definition}[section]
\newtheorem{cor}{Corollary}[section]
\newtheorem{eg}{Example}[section]
\numberwithin{equation}{section}

\newcommand{\Names}{Partha Rana, Sriparna Bandopadhyay}
\newcommand{\Title}{Construction of sign $k$-potent sign patterns and conditions for such sign patterns to allow $k$-potence}

\def\sgn{{\rm sgn}}

\def\min{{\rm min}}
\def\max{{\rm max}}

\def\red{{\rm red}}
\def\lcm{{\rm lcm}}

\begin{document}	
	%Insert the title of the paper
	\title{\Title\thanks{%Received by the editors on \DoS. Accepted for publication on \DoA. Handling Editor: \HE. 
			Corresponding Author: Partha Rana.}
	\author{Partha\ Rana\thanks{Department of Mathematics, Indian Institute of Technology Guwahati, Guwahati,
			Assam 781039, India.\\
            E-mail addresses: r.partha@iitg.ac.in (P. Rana), sriparna@iitg.ac.in (S. Bandopadhyay).}
\and Sriparna Bandopadhyay .}
}

	\markboth{\Names}{\Title}

\date{}

\maketitle

\begin{abstract}
  A sign pattern is a matrix whose entries are from the set $\{+,-, 0\}$. A square sign pattern $A$ is called sign $k$-potent if $k$ is the smallest positive integer for which $A^{k+1}=A$, and for $k=1$, $A$ is called sign idempotent. In 1993, Eschenbach \cite{01} gave an algorithm to construct sign idempotent sign patterns.
        However, Huang \cite{02} constructed an example to show that matrices obtained by Eschenbach's algorithm were not necessarily sign idempotent.
			In \cite{03}, Park and Pyo modified Eschenbach's algorithm to construct all reducible sign idempotent sign patterns. In this paper, we give an example to establish that the modified algorithm by  Park and Pyo does not always terminate in a single iteration; the number of iterations, depending on the order of the sign pattern, could be large. In this paper, we give a new algorithm that terminates in a single iteration to construct all possible sign idempotent sign patterns.  We also provide an algorithm for constructing sign $k$-potent sign patterns. Further, we give some necessary and sufficient conditions for a sign $k$-potent sign pattern to allow $k$-potence.
\end{abstract}


	\maketitle

	\section{Introduction}
An $n\times n$ matrix whose entries belong to the set $\{+,-,0\}$ is called a \textit{sign pattern matrix} or a \textit{sign pattern}. We denote the set of all $n\times n$ sign patterns by $\mathcal{Q}_n$. If $A=[a_{ij}]\in \mathcal{Q}_n$, then 
$$\mathcal{Q}(A)=\{B=[b_{ij}]\in\mathbb{R}^{n\times n}|~ \sgn(b_{ij})=a_{ij}~\text{for}~\text{all}~i,j=1,2,...,n\}$$ is called the \textit{qualitative class} of $A$.

 The addition $(+)+(-)$ is called an \textit{ambiguous} entry, and it is denoted by $\#$. Let $A=[a_{ij}]$ and $B=[b_{ij}]$ be two sign patterns of order $n$. The sum $A+B$ is said to be defined unambiguously if $a_{ij}b_{ij}\neq -$ for all $i,j=1,2,\dots,n$. The product $AB$ is defined unambiguously if, for every $i$ and $j$, the sum $\sum_{k=1}^{n}a_{ik}b_{kj}$ does not contain a pair of opposite signed terms.
 
A square sign pattern $A\in \mathcal{Q}_n$ is called \textit{sign $k$-potent} if $k$ is the smallest positive integer such that $A^{k+1}=A$. In particular, when $k=1$, the matrix $A$ is called \textit{sign idempotent}.
The class of sign idempotent sign patterns is closed under the following operations (Lemma~1.1, \cite{01}):
\begin{itemize}
	\item signature similarity,
	\item permutation similarity,
	\item transposition.
\end{itemize}
The following theorem is by Eschenbach \cite{01}.

\begin{theorem}[Theorem 1.4, \cite{01}] \label{nth1.1}
	If $A$ is an n-by-n $(n \geq 2)$ irreducible sign pattern matrix, then $A$ is sign idempotent if and only if it is entry wise positive.
\end{theorem}

Suppose that $A$ is an $n \times n$ real matrix. If there exists a permutation matrix $P$ such that
$$
P^{T} A P =
\begin{bmatrix}
    A_{11} & A_{12} \\
    0      & A_{22}
\end{bmatrix},
$$
where $A_{11}$ and $A_{22}$ are non vacuous square matrices, then $A$ is said to be \textit{reducible}. 
If no such permutation matrix exists, then $A$ is called \textit{irreducible}. 
Let $A\in \mathcal{Q}_n$ be a reducible sign pattern. By the Frobenius theorem, there exists a permutation matrix $P$ such that
\begin{equation}\label{eq2.1}
	P^{T}AP=\begin{bmatrix}
		A_{11}&A_{12}&\cdots&A_{1m} \\
		0&A_{22}&\cdots&A_{2m} \\
		\vdots&&\ddots&\vdots \\
		0&\cdots&&A_{mm}
	\end{bmatrix},
\end{equation}
where each $A_{ii}$ is an $n_i \times n_i$ irreducible matrix or a $1 \times 1$ zero block, and $P^{T} A P$ is called a \textit{Frobenius normal form} of $A$. 
Therefore, if $A$ is a sign idempotent sign pattern, then by Theorem~\ref{nth1.1}, each nonzero diagonal block of $A$ is an entry wise positive sign pattern.

Suppose that $A$ is in Frobenius normal form and $A$ is permutation similar to
$$	\begin{bmatrix}
	0_{h\times h}&0_{h\times p} \\
	0_{p\times h}&B_{p\times p}
\end{bmatrix},$$
for some positive integers $h$ and $p$. Then it is clear that $A$ is sign $k$-potent if and only if $B$ is sign $k$-potent. The zero rows and columns of $A$ are called \textit{extraneous zero rows and columns}. Therefore, we will assume that all matrices under consideration do not have any extraneous zero rows or columns.

Let $R$ denote a property that may or may not be satisfied by a real matrix.
A sign pattern $A$ \textit{requires} $R$ if every matrix $B \in \mathcal{Q}(A)$ satisfies $R$, and to \textit{allow} $R$, if there exists at least one matrix $B \in \mathcal{Q}(A)$ satisfies $R$. In particular, $A$ allows idempotence if there exists $B\in \mathcal{Q}(A)$ such that $B^2=B$.

If $A_1$ and $A_2$ are two sign patterns, then the patterns are called \textit{equivalent} if one can be produced from the other through a series of permutation similarity, signature similarity, negation, and transposition. Being equivalent, $A_1$ allows idempotence if and only if $A_2$ allows idempotence.

Given two sign patterns $A_1, A_2\in \mathcal{Q}_n$, we say $A_1$ is a \textit{subpattern} of $A_2$ if $A_2$ is obtained from $A_1$ by replacing some or none, of the zero entries of $A_1$ with either $+$ or $-$, and $A_2$ is called a \textit{superpattern} of $A_1$.

%\textcolor{red}{(DELETE)}The \textit{directed graph} $D(A)$ of an $n\times n$ sign patten $A=[a_{ij}]$, is the directed graph with $n$ vertices $\{1,2,...,n\}$, such that there is a directed edge in $D(A)$ from $i$ to $j$, denoted by $(i,j)$, if and only if $a_{ij}\neq 0$.   

Eschenbach~\cite{01} proposed an algorithm for constructing sign idempotent sign patterns. However, Huang~\cite{02} later provided an example showing that a sign pattern produced by this algorithm need not be sign idempotent. In 2011, Park and Pyo~\cite{03} modified Eschenbach’s algorithm and claimed that all sign idempotent sign patterns can be obtained using their modified version. In this paper, we present an example demonstrating that the modified algorithm of Park and Pyo~\cite{03} does not necessarily terminate in a single iteration.
 
In Section~\ref{s2}, we present the reduced form of a block sign pattern together with its cyclic normal form. We also review basic definitions and previously established results on sign $k$-potent sign patterns and their reduced forms. In Section~\ref{s3}, we present an algorithm to construct all possible idempotent sign patterns. Using Theorems~\ref{th11} and~\ref{th14}, we show that this algorithm terminates in a single iteration.  
 
   Eschenbach \cite{01} proposed the problem to characterize all sign $k$-potent sign patterns in terms of the diagonal and off-diagonal blocks of a sign pattern. In \cite{05}, Stuart et al. completely characterized the structure of irreducible sign $k$-potent sign patterns, and in \cite{06}, Stuart gave necessary results characterizing the structure of each of the off-diagonal blocks of a sign $k$-potent sign pattern.  In Section~\ref{s4}, we provide an algorithm that constructs all sign $k$-potent sign patterns in cyclic normal form.

A sign pattern $A\in \mathcal{Q}_n$ may be sign idempotent even though it does not allow idempotence. For example, 
$$A=\begin{bmatrix}
    +&-\\0&+
\end{bmatrix},$$
is a sign idempotent sign pattern, but it does not allow idempotence. 
In \cite{01}, Eschenbach proposed the problem of characterizing those sign idempotent sign patterns that allow idempotence. 
Subsequently, Huang~\cite{02} established several properties of sign idempotent sign patterns that allow idempotence, and Lee and Park~\cite{04} derived necessary and sufficient conditions for a sign idempotent sign pattern to allow idempotence. 
In Section~\ref{s5}, we consider sign $k$-potent sign patterns, identify several necessary and sufficient conditions under which a sign $k$-potent sign pattern in cyclic normal form allows $k$-potence.

%\{\{In Section~\ref{s5}, we consider sign $k$-potent sign patterns that allow $k$-potence. Huang~\cite{02} established several properties of sign idempotent sign patterns that allow idempotence, and Lee and Park~\cite{04} obtained necessary and sufficient conditions for a sign idempotent sign pattern to allow idempotence. In this section, we provide necessary and sufficient conditions for a sign $k$-potent sign pattern in cyclic normal form to allow $k$-potence.\}\}
% and we prove that a sign pattern $A$ allows $k$-potence if and only if $red(A)$ allows $k$-potence.

%%%%%%%%%%%%%%%%%%%%%%%%%%%%%%%%%%%%%%%%%%%%%%%%%%%%%%%%%
\section{Reduced form of a block matrix and cyclic normal form.} \label{s2}
%In \cite{05}, Stuart et al. defined two sign pattern matrices, $P_n$ and $Q_n$. $P_n$ denotes the $n\times n$ circulant permutation pattern with $+$'s on the first super diagonal and $+$ in the $(n,1)$-th position, and $Q_n$ denotes the sign pattern matrix obtained from $P_n$ by replacing the $+$ in the $(n,1)$-th position with a $-$. Apparently, $P_1=[+]$ and $Q_1=[-]$, the signed $1\times 1$ pattern matrices. Then $P_n$ is a sign n-potent matrix and $Q_n$ is a sign 2n-potent matrix.

All sign patterns admit a symmetric block partition such that each block is of the form $\alpha J$, where $\alpha \in \{+,-,0\}$ and $J$ is the all $+$ matrix of appropriate size. The block partition of a sign pattern that is not a proper subpartition of any other symmetric block partition is called a \textit{coarsest block partition}. From Lemma 9 of \cite{05}, the coarsest block partition of a sign pattern is unique.

The \textit{reduced block matrix} for a sign pattern $A$, denoted by $\red(A)$, is the matrix whose entries are induced by the coarsest block partition of $A$. That is, if the $(i,j)$-th block of the coarsest block partition of $A$ is of the form $\alpha_{ij}J_{n_i\times n_j}$ for $1\leq i,j \leq m$ for some $m$, then $\red(A)$ is an $m\times m$ matrix whose $(i,j)$-th entry is $\alpha_{ij}$, where $m$ is the number of diagonal blocks of $A$.

Stuart et al. \cite{05} gave the following result.

\begin{theorem}[Theorem 10, \cite{05}] \label{2.1a}
	Let $A$ be a square generalized sign pattern. Then for each positive integer $k$, $\red(A^k)=\red([\red(A)]^k)$. Further, if $A^{k+1}=A^k$ for some positive integer $k$, then $[\red(A)]^{k+1}=\red(A)$.
\end{theorem}

In \cite{05}, Stuart et al.\ introduced two sign patterns, denoted by $P_n$ and $Q_n$. The matrix $P_n$ is the $n \times n$ circulant permutation sign pattern with $+$ entries on the first superdiagonal and in the $(n,1)$ position, and $Q_n$ is obtained from $P_n$ by replacing the $+$ in the $(n,1)$ position with a $-$. Apparently, $P_1=[+]$ and $Q_1=[-]$, the signed $1\times 1$ patterns. Then $P_n$ is a sign $n$-potent matrix and $Q_n$ is a sign $2n$-potent matrix.
In addition, Stuart et al. \cite{05} gave the following characterization for irreducible sign $k$-potent sign patterns.
\begin{theorem}[Theorem 11, \cite{05}]
	Let $A$ be an irreducible sign pattern. Then $A$ is sign $k$-potent if and only if $A$ can be transformed via signature similarity and permutation similarity into a pattern $B$ such that $\red(B)$ is either $P_m$ or $Q_m$ for some $m$. If $k$ is odd, then $m=k$ and $\red(B)=P_k$. If $k$ is even, then either $m=k$ and $\red(B)=P_k$, or else $m=k/2$ and $\red(B)=Q_{k/2}$.
\end{theorem}

Suppose that $A$ is a reducible sign pattern in Frobenius normal form (\ref{eq2.1}) with $n$ diagonal blocks, where the irreducible diagonal block $A_{ii}$ is either a $1\times 1$ zero block or else a matrix such that $\red(A_{ii})$ is in $\{P_{m_i},~Q_{m_i}\}$ for some $m_i$, then this form of $A$ is called a \textit{cyclic normal form} of $A$.
For reducible sign patterns, the following results were given by Stuart in \cite{06}.
%\begin{defn} \label{d1}
%	Suppose that $A$ is a reducible sign pattern in Frobenius normal form (\ref{eq2.1}) with $n$ diagonal blocks, where the irreducible diagonal block $A_{ii}$ is either a $1\times 1$ zero block or else a matrix such that $red(A_{ii})$ is in $\{P_{m_i},~Q_{m_i}\}$ for some $m_i$, then this form of $A$ is called a \textit{cyclic normal form} of $A$.
%\end{defn}

\begin{lemma}[Lemma 7, \cite{06}] \label{11}
	Let $A$ be in cyclic normal form. Let $A$ be subblock partitioned by the partition induced by the cyclic blocks of each irreducible diagonal block of the Frobenius normal form. If $A^{k+1}=A$ for some $k\geq1$, then for all $i< j$, the subblocks $A_{ij}$ is given by $A_{ij}=\alpha_{ij}J_{ij}$ where $\alpha_{ij}\in \{0,+,-\}$ and $J_{ij}$ is the matrix of the same size as $A_{ij}$ all of whose entries are pluses.\end{lemma}

\begin{theorem}[Theorem 9, \cite{06}] \label{2.3aa}
	Let $A$ be in cyclic normal form with no extraneous zero rows and columns. Let $A$ be subblock partitioned by the partition induced by the cyclic blocks of each irreducible, diagonal block of the cyclic normal form. The following are equivalent:
	\begin{itemize}
		\item[i.] The matrix $A^{k+1} = A$ for some positive integer k;
		\item[ii.]  For all $i$ and $j$, the subblock $A_{ij} = \alpha_{ij} J_{ij}$ where $\alpha_{ij}\in \{0, +, -\}$; and the matrix of coefficients, $\red(A)$, satisfies $[\red(A)]^{k+1} = \red(A)$ for some positive integer $k$.
	\end{itemize}
	
\end{theorem}
%%%%%%%%%%%%%%%%%%%%%%%%%%%%%%%%%%%%%%%%%%%%%%%
\section{Algorithm to construct sign idempotent sign patterns.} \label{s3}

Eschenbach in \cite{01}, gave an algorithm for constructing all reducible sign idempotent sign patterns. Also, asserting that all sign idempotent sign patterns can be constructed by using the algorithm. However, in 2008, Huang \cite{02} gave the following example of a sign pattern which was constructed by the Eschenbach algorithm but is not sign idempotent.
$$A=\begin{bmatrix}
	+&+&-&+ \\
	0&+&0&+\\
	0&0&0&+\\
	0&0&0&+
\end{bmatrix}.$$

Park and Pyo in \cite{03}, gave a modified version of Eschenbach's algorithm to construct all reducible sign idempotent sign patterns and asserted that all sign idempotent sign patterns can be constructed by using the Modified algorithm.

However, the \textit{Modified algorithm} does not always terminate in a single iteration, as the following example illustrates.
\begin{equation} \label{2.22}
	A=\left[
	\begin{array}{c|c|c|c|c}
		+&-&-&-&\\\hline
		&0&0&+&+\\\hline
		&&+&0&-\\\hline
		&&&+&+\\\hline
		&&&&+
	\end{array}
	\right].
\end{equation}
The blocks in the first three super diagonals of $A$ are obtained by using the Modified algorithm, but it is not possible to construct the block in the fourth super diagonal by using the Modified algorithm.

In this section, we present a new algorithm to construct all sign idempotent sign patterns and show that it terminates in a single iteration.
The following lemmas in \cite{02} determine the structure of the upper diagonal blocks of a reducible sign idempotent pattern in Frobenius normal form.

\begin{lemma}[Lemma 1.1, \cite{02}]
	Suppose $A$ is an $n\times n$ reducible sign pattern in Frobenius normal form (\ref{eq2.1}). If $A_{ii}$ and $A_{jj}$ are positive blocks, then $A$ is sign idempotent only if $A_{ij}$ is constantly signed.
\end{lemma}
\begin{lemma}[Lemma 1.2, \cite{02}]
	Suppose $A$ is an $n\times n$ reducible sign pattern in Frobenius normal form (\ref{eq2.1}). If $A_{ii}$ is positive and $A_{jj}=(0)$, then $A$ is sign idempotent only if $A_{ij}$ is constantly signed.
\end{lemma}
\begin{lemma}[Lemma 1.2(ii), \cite{02}]
	Suppose $A$ is an $n\times n$ reducible sign pattern in Frobenius normal form (\ref{eq2.1}). If $A_{ii}=(0)$ and $A_{jj}$ is positive, then $A$ is sign idempotent only if $A_{ij}$ is constantly signed.
\end{lemma}

Therefore, if $A$ is a reducible sign idempotent sign pattern in Frobenius normal form [\ref{eq2.1}] with $n$ diagonal blocks, then by using the above lemmas we have $A_{ij}=\alpha_{ij}J$ for $1\leq i,j\leq n$, $\alpha_{ij}\in \{+,-,0\}$. Thus, if we take the block partition of $A$ according to the Frobenius normal form of $A$, that block partition of $A$ forms the coarsest block partition of $A$. According to Theorem \ref{2.3aa}, $A$ is sign idempotent if and only if $\red(A)$ is sign idempotent. Therefore, we may assume that $A$ is in reduced form.

Next, we give an algorithm for constructing all sign idempotent sign patterns.\\
\textbf{Algorithm:} Let $A$ be a reducible sign pattern in Frobenius normal form with $n$ diagonal blocks and suppose that $A$ is in reduced form. Then all the $1\times 1$ diagonal blocks of $A$ are either zero or $+$. Here we recursively define the sign of the off diagonal positions of $A$ column wise, so that $A$ is sign idempotent.
$$A =
\left[
\begin{array}{c|c|c|c|c|c|c|c|c}
	\ddots &  &  & & &&&& \\
	\hline
	& A_{ll} & \cdots & A_{li}&A_{l,i+1}&\cdots&A_{l,j-1}&& \\
	\hline
	&& \ddots & \vdots&\vdots&&\vdots&& \\
	\hline
	&&&A_{ii}&A_{i,i+1}&\cdots&A_{i,j-1}&A_{ij}&\\
	\hline
	&&&&A_{i+1,i+1}&\cdots&\cdots&A_{i+1,j}&\\
	\hline
	&&&&&\ddots&&\vdots&\\
	\hline
	&&&&&&A_{j-1,j-1}&A_{j-1,j}&\\
	\hline
	&&&&&&&A_{jj}&\\
	\hline
	&&&&&&&&\ddots\\
	
\end{array} \right].$$

\begin{itemize}
	\item[1.] Determine $A_{j-1,j} $ in the $j$-th column as follows:
	\subitem I. if both $A_{j-1,j-1}$ and $A_{jj}$ are zero, then take $A_{j-1,j}=0$,
	\subitem II. if at least one of $A_{j-1,j-1}$, $A_{jj}$ is nonzero, then take $A_{j-1,j}$ arbitrarily from $\{+,-,0\}$.\\
	Move up to the above entry in the same column if there is one.
	\item[2.] For $A_{ij}$, $i=j-2,j-3,...,1$, let  $\sum_{p=i+1}^{j-1}A_{ip}A_{pj}=X_{{ij}}$ and $\sum_{p=i+1}^{j-1}A_{lp}A_{pj}=X_{{ij}}^l$,  $l=i-1,i-2,...,1$.
	\subitem I. If $X_{{ij}}$ is nonzero, then take $A_{ij}=X_{ij}$.
	\subitem II. If $X_{{ij}}$ is zero and,
	\subsubitem i. if both $A_{ii}$ and $A_{jj}$ are zeros, then take $A_{ij}$ as zero,
	\subsubitem ii. if at least one of $A_{ii}$, $A_{jj}$ is nonzero, then take $A_{ij}$ such that,
	\subsubitem ~~~~a. if $A_{li}X_{{ij}}^l=0$ for all $l=i-1,i-2,...,1$, take $A_{ij}$ arbitrarily,
	\subsubitem ~~~b. if there exists an $l$, $l=i-1,i-2,...,1$ such that $A_{li}X_{{ij}}^l\neq 0$, take $A_{ij}$ such that either $A_{li}A_{ij}=X_{{ij}}^l$ or $A_{ij}=0$.\\
	When all the blocks of the $j$-th column are specified, move to the next column by increasing $j$ by 1, if there is one.
	
\end{itemize}
The following example illustrates how the above algorithm works. 
\begin{eg}\rm
	Suppose $A$ is a reducible block sign pattern in Frobenius normal form. Let $A$ be in reduced form and let diagonal blocks of $A$ be taken as
	$$A =
	\left[
	\begin{array}{c|c|c|c|c}
		+ &  &  & &   \\
		\hline
		& 0 &  &  &   \\
		\hline
		&& + & &  \\
		\hline
		&&&+&\\
		\hline
		&&&&+\\
	\end{array} \right].$$
	Determine the superdiagonal blocks of $A$ by using the algorithm, so that $A$ becomes sign idempotent. For the blocks $A_{12}$ and $A_{23}$, apply step~(1) of the algorithm and take $A_{12}$ and $A_{23}$ as
	$$A =
	\left[
	\begin{array}{c|c|c|c|c}
		+ & - &  & &   \\
		\hline
		& 0 &0  &  &   \\
		\hline
		&& + & &  \\
		\hline
		&&&+&\\
		\hline
		&&&&+\\
	\end{array} \right].$$ Since $X_{{13}}=[0]$. So by step~(2)-(II)-(ii), take $A_{13}=[-]$. Also, by applying step~(1), take $A_{34}=[0]$. Therefore,
	$$A =
	\left[
	\begin{array}{c|c|c|c|c}
		+ & - & - & &   \\
		\hline
		& 0 &0  &  &   \\
		\hline
		&& + & 0&  \\
		\hline
		&&&+&\\
		\hline
		&&&&+\\
	\end{array} \right].$$
	Since $X_{{24}}=[0]$ and $X_{{24}}^1=[0]$, hence by step~(2)-(II)-(ii), take $A_{24}=[+]$. Also, $X_{{14}}=[-]$, hence by step~(2)-(I), $A_{14}=[-]$. Therefore,  
	$$A =
	\left[
	\begin{array}{c|c|c|c|c}
		+ & - & - &- &   \\
		\hline
		& 0 &0  &+  &   \\
		\hline
		&& + & 0&  \\
		\hline
		&&&+&\\
		\hline
		&&&&+\\
	\end{array} \right].$$
	By step~(1), we can take $A_{45}=[+]$. Since $X_{{35}}=[0]$, $X_{{35}}^1=A_{14}A_{45}=[-]$ and $A_{13}=[-]$, by step~(2)-(II)-(ii), $A_{35}=[+]$. Again, $X_{{25}}=[+]$, by step~(2)-(I), $A_{25}=[+]$. Then $X_{15}=[-]$ and by step~(2)-(I), $A_{15}=[-]$. Therefore,  $$A =
	\left[
	\begin{array}{c|c|c|c|c}
		+ & - & - &- & -  \\
		\hline
		& 0 &0  &+  &+   \\
		\hline
		&& + & 0& + \\
		\hline
		&&&+&+\\
		\hline
		&&&&+\\
	\end{array} \right],$$ is obtained from algorithm and $A^2=A$.
	
\end{eg}

Next, we show through a series of results that the matrices obtained by the above algorithm are sign idempotent, and conversely, that every sign idempotent sign pattern can be obtained by using the above algorithm in a single iteration.

First, we show that any matrix obtained by the above algorithm is sign idempotent. The following theorem gives the result when the number of diagonal blocks is equal to 2. 

\begin{theorem} \label{th11}
	Let  $A=\begin{bmatrix}
		A_{11}&A_{12} \\
		0&A_{22}
	\end{bmatrix}$ be a reducible sign pattern in Frobenius normal form. Suppose that $A$ is in reduced form and $A$ is obtained by the algorithm, then $A$ is sign idempotent.
\end{theorem}
\begin{proof}
	Since $A$ is obtained by the algorithm, if both $A_{11}$ and $A_{22}$ are zeros, then $A_{12}=0$. Therefore, $A$ is a zero matrix, so $A$ is idempotent.
	If any one of $A_{11}$, $A_{22}$ is nonzero, then by step~(1)-(II) of the algorithm $A_{12}$ is taken arbitrarily. Since $A_{11}$, $A_{22}$ are either zero or +, $A_{11}A_{12}+A_{12}A_{22}=A_{12}$ and hence $A$ is sign idempotent.
\end{proof}

%The following results establish that the proposed algorithm terminates with a sign idempotent sign pattern matrix.

%\begin{lemma}\label{0.3}
%	Let $A$ be a reducible sign pattern matrix in Frobenius normal form with $m$ diagonal blocks and suppose that $A$ is in reduced form. Assume that all the blocks in the first $j-1$ columns and the blocks $A_{j-1,j},A_{j-2,j},...,A_{i+1j}$ in the $j$-th column of $A$ where $1<i \leq j-1$, $3\leq j\leq m$ are obtained by the algorithm. Then $\sum_{p=i}^{j-1}A_{lp}A_{pj}$ is defined unambiguously for all $1\leq l< i$.
%\end{lemma}

The following theorem shows that the patterns obtained by the above algorithm are sign idempotent when the number of diagonal blocks is more than $2$.

\begin{theorem}\label{3.3a}
	Let $A$ be a reducible sign pattern in Frobenius normal form with $n$ diagonal blocks and suppose that $A$ is in reduced form. Assume that the first $j$ columns of $A$, $2\leq j\leq n$, are obtained by the algorithm. Then the $j\times j$ submatrix of $A$ is formed by the first $j$ columns and the first $j$ rows of $A$ are sign idempotent.
\end{theorem}
\begin{proof}
	Let $A$ be a reducible sign pattern in reduced form, and let all the entries of $A$ up to the $j$-th column be obtained by the algorithm. So, we consider the $(i,j)$-th entry of $A$ for $i=j-1,j-2,...,1$.
	To show $$A_{ij}=\sum_{p=i}^{j}A_{ip}A_{pj}.$$
		For $i=j-1$, since $A_{jj}, ~A_{j-1,j-1}$ are either zero or $+$, by step~(1) of the algorithm $A_{j-1,j}=A_{j-1,j-1}A_{j-1,j}+A_{j-1,j}A_{jj}$. 
	Now we consider $A_{ij}$ for $i=j-2,...,1$.
	
	\textbf{Case~1:} $A_{ij}$ is obtained by step (2)-(I) of the algorithm. 
	\\Therefore, $\sum_{p=i+1}^{j-1}A_{ip}A_{pj}$ is nonzero and $A_{ij}=\sum_{p=i+1}^{j-1}A_{ip}A_{pj}$. Thus,
	$A_{ii}A_{ij}=A_{ii}\sum_{p=i+1}^{j-1}A_{ip}A_{pj}=\sum_{p=i+1}^{j-1}A_{ip}A_{pj}$ or $A_{ii}A_{ij}=0$ (since $A_{ii}=+$ or $0$).
	Similarly, $A_{ij}A_{jj}=\sum_{p=i+1}^{j-1}A_{ip}A_{pj}$ or $A_{ij}A_{jj}=0$, so $\sum_{p=i}^{j}A_{ip}A_{pj}=A_{ij}$.
	
	\textbf{Case~2:} $A_{ij}$ is obtained by step (2)-(II)-(i) of the algorithm.
	\\Therefore, $\sum_{p=i+1}^{j-1}A_{ip}A_{pj}=0$ and both $A_{ii}$, $A_{jj}$ are equal to zero, so by the algorithm $A_{ij}=0$ and $\sum_{p=i}^{j}A_{ip}A_{pj}=A_{ij}$.
	
	\textbf{Case~3:} $A_{ij}$ is obtained by step (2)-(II)-(ii) of the algorithm.
	\\Therefore, $\sum_{p=i+1}^{j-1}A_{ip}A_{pj}=0$ and at least one of $A_{ii}$, $A_{jj}$ is nonzero. Hence, $\sum_{p=i}^{j}A_{ip}A_{pj}=A_{ii}A_{ij}+A_{ij}A_{jj}=A_{ij}$.
\end{proof}

The following result shows that the above algorithm terminates in a single iteration.

\begin{lemma}\label{0.3}
	Let $A$ be a reducible sign pattern in Frobenius normal form with $n$ diagonal blocks and suppose that $A$ is in reduced form. Assume that all the blocks in the first $j-1$ columns and the blocks $A_{j-1,j},A_{j-2,j},...,A_{i+1,j}$ in the $j$-th column of $A$ where $1\leq i \leq j-2$, $3\leq j\leq n$ are obtained by the algorithm. Then $A_{ij}$ can be obtained by the above algorithm in the next step.
\end{lemma}

\begin{proof} In order to prove the above lemma, it is enough to show that $X_{ij}$, $X_{ij}^l$ are unambiguously defined, that is  $\sum_{p=i+1}^{j-1}A_{lp}A_{pj}$ is defined unambiguously for all $1\leq l\leq i$.
	We prove this lemma by using induction on $i$, where $i+1$ varies from $j-1$ to $2$.
	If $i+1=j-1$, then $A_{j-1,j}$ is obtained by step~(1), so $A_{l,j-1}A_{j-1,j}$ is defined unambiguously for all $1\leq l\leq i$.
	Now we assume that the result is true for $i+1\geq k+1$, so for $i+1=k+1$, $\sum_{p=k+1}^{j-1}A_{lp}A_{pj}$ is defined unambiguously for all $1\leq l \leq k$. 
	
	For $i+1=k$, if $A_{kj}$ is obtained by step~(2)-(II) of the algorithm, then $X_{kj}, X_{kj}^l$ are both unambiguously defined and the result follows. Therefore, assume that $A_{kj}$ is obtained by step~(2)-(I) of the algorithm. Since
	$$\sum_{p=k}^{j-1}A_{lp}A_{pj}= A_{lk}A_{kj}+\sum_{p=k+1}^{j-1}A_{lp}A_{pj},$$ for all $1\leq l \leq k-1$, and $A_{kj}$ is obtained by step~(2)-(I) of the algorithm,
	therefore $A_{kj}=\sum_{p=k+1}^{j-1}A_{kp}A_{pj}\neq 0$. So, there exists at least one $p$, $k+1\leq p\leq j-1$, such that $A_{kj}=A_{kp}A_{pj}\neq 0$, which implies $A_{lk}A_{kj}=A_{lk}A_{kp}A_{pj}$.
	Again, $l<k+1\leq p\leq j-1$, so by step~(2)-(I) of the algorithm $A_{lk}A_{kp}$ is a subpattern of $A_{lp}$, so $A_{lk}A_{kj}$ is a subpattern of $A_{lp}A_{pj}$ for some $p\geq k+1$. Hence $A_{lk}A_{kj}$ is a subpattern of $\sum_{p=k+1}^{j-1}A_{lp}A_{pj}$ and $\sum_{p=k}^{j-1}A_{lp}A_{pj}$ is unambiguously defined for all $l$, $1\leq l\leq k-1$.
\end{proof}

Theorem \ref{3.3a} and Lemma \ref{0.3} together imply that the algorithm terminates in a single iteration and produces a sign idempotent sign pattern. The following theorem establishes that every sign idempotent sign pattern can be obtained by the algorithm.

\begin{theorem} \label{th14}
	Let $A$ be a reducible sign pattern in Frobenius normal form and suppose that $A$ is in reduced form. If $A$ is sign idempotent sign pattern, then all the off diagonal blocks of $A$ can be obtained by the algorithm.
\end{theorem}
\begin{proof}
	Since $A$ is sign idempotent, $A_{ij}=A_{ii}A_{ij}+\sum_{p=i+1}^{j-1}A_{ip}A_{pj}+A_{ij}A_{jj}$ for all $i=j-1,j-2,...,1$, $j=2,3,...,n$. Hence $A_{j-1,j}=A_{j-1,j-1}A_{j-1,j}+A_{j-1,j}A_{jj}$.
	
	If both $A_{j-1,j-1}$, $A_{jj}$ are equal zero, then $A_{j-1,j}=0$. If at least one of $A_{j-1,j-1}$, $A_{jj}$ is nonzero, then $A_{j-1,j}$ is from $\{0,-,+\}$. So, in each case, the block $A_{j-1,j}$ can be obtained by step~(1) of the algorithm.
	
	Now consider $A_{ij}$ for $i=j-2,j-3,...,1$.
	If $\sum_{p=i+1}^{j-1}A_{ip}A_{pj}\neq 0$, then $$A_{ij}=A_{ii}A_{ij}+\sum_{p=i+1}^{j-1}A_{ip}A_{pj}+A_{ij}A_{jj}=\sum_{p=i+1}^{j-1}A_{ip}A_{pj},$$ therefore $A_{ij}$ can be obtained by step~(2)-(I) of the algorithm.
	If $\sum_{p=i+1}^{j-1}A_{ip}A_{pj}=0$ and both $A_{ii}$ and $A_{jj}$ are equal to zero, then $A_{ij}=\sum_{p=i}^{j}A_{ip}A_{pj}=0$, so A$_{ij}$ can be obtained by step (2)-(II)-(i).
	If $\sum_{p=i+1}^{j-1}A_{ip}A_{pj}=0$ and at least one of $A_{ii}$, $A_{jj}$ is nonzero, then $A_{ij}=A_{ii}A_{ij}+A_{ij}A_{jj}$, therefore $A_{ij}$ can be any of $0,+$ or $-$ and can be obtained by step~(2)-(II)-(ii) of the algorithm. 
	So all the elements of the $j$-th column of $A$ can be obtained by the algorithm.
\end{proof}
%%%%%%%%%%%%%%%%%%%%%%%%%%%%%%%%%%%%%%%%%%%%%%%%%%%%%%%%%%%%
\section{Algorithm for sign $k$-potent sign patterns}\label{s4}
In this section, we give an algorithm to construct all sign $k$-potent sign patterns in cyclic normal form.
Suppose that $A$ is a nonzero, reducible sign pattern in cyclic normal form with no extraneous zero rows and columns, then by Theorem \ref{2.3aa}, $A$ is sign $k$-potent if and only if $\red(A)$ is sign $k$-potent.
Therefore, we may assume that $A$ is in reduced form where the diagonal blocks are irreducible or $1\times 1$ zero blocks. So, the diagonal blocks of $A$ are either $1\times 1$ zero matrices or matrices of the type $P_m$ or $Q_m$ for some $m\in \mathbb{N}$.

Stuart \cite{06} established the following necessary conditions on the structure of each off-diagonal block of a sign 
$k$-potent sign pattern $A$.

\begin{lemma}[Lemma 11, \cite{06}] \label{4.55}
	Let $A$ be a block upper triangular sign pattern matrix with diagonal blocks $A_{ii}$. Suppose $j>i$. If $A^{k+1}=A$ for some positive integer $k$, then $(A_{ii})^{k-h}A_{ij}(A_{jj})^h$ is a subpattern of $A_{ij}$ for $0\leq h \leq k$. Further, if $A_{ii}$ is of type $P_m$ or $Q_m$ and if $A_{jj}$ is of type $P_n$ or $Q_n$, then $A_{ii}A_{ij}=A_{ij}A_{jj}$.
\end{lemma}

\begin{theorem}[Theorem 12, \cite{06}] \label{4.66}
	Let $A_{11}=P_m$ and $A_{22}=P_n$. Then $A_{11}B=BA_{22}$ if and only if $B$ decomposes as copies of the $g\times g$ circulant pattern matrix $C=\sum_{h=0}^{g-1}b_hP^h$ where $g=\gcd\{m,n\}$, $P=P_g$, and the $b$ is an arbitrary $g\times 1$ sign pattern vector. Specifically, $B$ block partitions with the $i,j$-block is given by $C$ for $1\leq i\leq m/g$ and $1\leq j\leq n/g$.
\end{theorem}

\begin{theorem}[Theorem 13, \cite{06}] \label{4.77}
	Let $A_{11}=Q_m$ and $A_{22}=Q_n$. Let $g=\gcd\{m,n\}$. Suppose $(m+n)/g$ is odd. Then $A_{11}B=BA_{22}$ if and only if $B=0$. Suppose that $(m+n)/g$ is even. Then $A_{11}B=BA_{22}$ if and only if $B$ decomposes as alternatingly signed copies of the $g\times g$ anticirculant pattern matrix $C=\sum_{h=0}^{g-1}b_hQ^h$ where $Q=Q_g$ and where $b$ is an arbitrary $g\times 1$ sign pattern vector. That is, $B$ block partitions with the $i,j$-block is given by $(-)^{i+j}C$ for $1\leq i\leq m/g$ and $1\leq j\leq n/g$.
\end{theorem}

\begin{theorem}[Theorem 14, \cite{06}] \label{4.88}
	Let one of $A_{11}$ and $A_{22}$ be equal to $Q_m$ and the other be equal to $P_n$. Let $g=\gcd\{m,n\}$. Suppose $n/g$ is odd. Then $A_{11}B=BA_{22}$ if and only if $B=0$. Suppose $n/g$ is even. Then $A_{11}B=BA_{22}$ if and only if $B$ decomposes as alternatingly signed copies of the $g\times g$ anticirculant pattern matrix $C=\sum_{h=0}^{g-1}b_hQ^h$ where $Q=Q_g$ and where $b$ is an arbitrary $g\times 1$ sign pattern vector. That is, $B$ block partitions with the $i,j$-block given by $(-)^{i+j}C$ for $1\leq i\leq m/g$ and $1\leq j\leq n/g$.
\end{theorem}

By using the above three theorems, Stuart \cite{06} characterized the off-diagonal blocks $A_{ij}$ of $A$, that is, if $A_{ii}$, $A_{jj}$ are two nonzero irreducible diagonal blocks of $A$ of order $m,n$, respectively, then $A_{ij}$ can be decomposed as copies of $g\times g$ circulant (or anticirculant) sign patterns $C=\sum_{h=0}^{g-1}b_hP_g^h$ (or $C=\sum_{h=0}^{g-1}b_hQ_g^h$), where $g=\gcd\{m,n\}$ and $b_h\in \{+,-,0\}$.
Thus, $A_{ij}$ can be written as copies of circulant or anticirculant sign patterns, as given in the following tables.

\begin{equation}\label{4.1}
	\begin{tabular}{||c c c||} 
		\hline
		$A_{ii}$ & $A_{jj}$ & $g\times g$ subblocks of $A_{ij}$\\ [0.5ex] 
		\hline\hline
		$P_m$ & $P_n$ & Circulant sign pattern \\ 
		\hline
		$P_m$ & $Q_n$ & $m/g$ odd $\implies$ $0_{g\times g}$,\\
		&&    $m/g$ even$\implies$ $\pm{(anticirculant~sign~pattern)}$   \\
		\hline
		$Q_m$ & $P_n$ &  $n/g$ odd $\implies$ $0_{g\times g}$,\\
		&&    $n/g$ even$\implies$ $\pm{(anticirculant~sign~pattern)}$  \\
		\hline
		$Q_m$ & $Q_n$ & \hspace{.2cm} $(m+n)/g$ odd $\implies$ $0_{g\times g}$,\\
		&&    \hspace{.5cm} $(m+n)/g$ even$\implies$ $\pm{(anticirculant~sign~pattern)}$  \\ [1ex] 
		
		\hline
		
	\end{tabular}
\end{equation}
\begin{equation}\label{4.1.1}
	\begin{tabular}{||c c c||} 
		\hline
		$A_{ii}$ & $A_{jj}$ &   $A_{ij}$\\ [0.5ex]
        \hline \hline
        $0_{1\times 1}$ & $P_n$ or $Q_n$ & Sign pattern row vector \\ 
		\hline
		$P_m$ or $Q_m$ &  $0_{1\times 1}$ & ~~~~Sign pattern column vector \\
		\hline
		$0_{1\times 1}$ & $0_{1\times 1}$&  $0_{1\times 1}$ ~~~(if~ $j=i+1$~)\\
		\hline
        \end{tabular}
\end{equation}

%%%%%%%%%%%%%%%%%%%%%%%

Using the above characterization of sign $k$-potent sign patterns, we give an algorithm to construct sign $k$-potent sign patterns.

Suppose $A$ is a reducible sign pattern in cyclic normal form with $n$ diagonal blocks, and $A$ is in reduced form. If $A$ is sign $k$-potent, then any off-diagonal block $A_{ij}$ of $A$ for any $i,j$; $i\neq j$ should satisfy $$\sum_{i_1\leq i_2\leq \cdots \leq i_k=i}^{j}A_{ii_1}A_{i_1i_2}\cdots A_{i_kj}=A_{ij}.$$ Then the block $A_{ij}$ is said to satisfy the condition of $k$-potence.
\\
\textbf{Algorithm.} Let $A$ be a reducible sign pattern in cyclic normal form with $n$ diagonal blocks, and suppose that $A$ is in reduced form. Then all possible choices for the diagonal blocks of $A$ are $P_{k_i}$, $Q_{k_i/2}$ or $[0]$ for some $k_i\geq 1$ and let $k=\lcm\{k_i~| ~1\leq i \leq n\}$. Here we recursively define the off-diagonal blocks of $A$ column wise, so that $A$ is sign $k$-potent. 
$$A =
\left[
\begin{array}{c|c|c|c|c|c|c|c|c}
	\ddots &  &  & & &&&& \\
	\hline
	& A_{ll} & \cdots & A_{li}&A_{l,i+1}&\cdots&A_{l,j-1}&& \\
	\hline
	&& \ddots & \vdots&\vdots&&\vdots&& \\
	\hline
	&&&A_{ii}&A_{i,i+1}&\cdots&A_{i,j-1}&A_{ij}&\\
	\hline
	&&&&A_{i+1,i+1}&\cdots&\cdots&A_{i+1,j}&\\
	\hline
	&&&&&\ddots&&\vdots&\\
	\hline
	&&&&&&A_{j-1,j-1}&A_{j-1,j}&\\
	\hline
	&&&&&&&A_{jj}&\\
	\hline
	&&&&&&&&\ddots\\
	
\end{array} \right].$$ 

\begin{itemize}
	\item[1.] For the block $A_{j-1,j}$ in the j-th column, we define $A_{j-1,j}$ by using Tables (\ref{4.1}), (\ref{4.1.1}), such that
	
	$$\sum_{p=1}^{k}A_{ll}^{k-p}X_{j-1,j}^lA_{jj}^{p-1}+A_{l,j-1}A_{j-1,j-1}^{k-1}A_{j-1,j},$$ is defined unambiguously for all $l$, $1\leq l \leq j-2$, where $X_{j-1,j}^l=A_{l,j-1}A_{j-1,j}$.
	
	Move up to the above block in the same column if there is one.
	
	\item[2.] For the block $A_{ij}$ in the $j$-th column for $i=j-2,j-3,...,1$,\\
	let, $\sum_{r=i+1}^{j-1}A_{ir}A_{rj}=X_{ij}$ and $\sum_{r=i}^{j-1}A_{lr}A_{rj}=X_{ij}^l$ for $l=i-1,i-2,...,1$.
	
	\subitem I. If both $A_{ii}$ and $A_{jj}$ are zero blocks, then take $A_{ij}=\sum_{r=i+1}^{j-1}A_{ir}A_{rr}^{k-1} A_{rj}$.
	\subitem II. If at least one of $A_{ii}$, $A_{jj}$ is a nonzero block, then
	\subsubitem{i.} if $X_{ij}$ is an entry wise nonzero block, take $A_{ij}=\sum_{p=1}^{k}A_{ii}^{k-p}X_{ij}A_{jj}^{p-1}$,
	\subsubitem{ii.} if $X_{ij}$ contains at least one zero entry, take $A_{ij}$ by using Tables (\ref{4.1}), (\ref{4.1.1}), such that $\sum_{p=1}^{k}A_{ii}^{k-p}X_{ij}A_{jj}^{p-1}$ is a subpattern of $A_{ij}$ and 
	$$\sum_{p=1}^{k}A_{ll}^{k-p}X_{ij}^lA_{jj}^{p-1}+\sum_{r=i}^{j-1}A_{lr}A_{rr}^{k-1}A_{rj}$$ is defined unambiguously for all $l$, $l=i-1,i-2,...,1$.
	
	When all the blocks of the $j$-th column are specified, move to the next column by increasing $j$ by $1$ if there is one.
	
\end{itemize}

The following example shows how the algorithm works.

\begin{eg}\rm
	Suppose $A$ is a reducible, block sign pattern in cyclic normal form. Let $A$ be in reduced form and the diagonal blocks of $A$ are taken as,
	
	$$A =
	\left[
	\begin{array}{cc|c|cc|c}
		0 & + &  & & &  \\
		+ & 0 &  &  &  &  \\
		\hline
		&& 0 & & & \\
		\hline
		&&&0&+&\\
		&&&+&0&\\
		\hline
		&&&&&-\\
	\end{array} \right].$$ 
	
	Determine the super diagonal blocks of $A$ by using the algorithm, so that $A$ is sign 2-potent. For the blocks $A_{12}$ and $A_{23}$ we apply step (1) of the algorithm and take, 
	
	$$A =
	\left[
	\begin{array}{cc|c|cc|c}
		0 & + & + & & &  \\
		+ & 0 & - &  &  &  \\
		\hline
		&& 0 & +& -& \\
		\hline
		&&&0&+&\\
		&&&+&0&\\
		\hline
		&&&&&-\\
	\end{array} \right].$$ 
	
	Since $A_{12}A_{23}=\begin{bmatrix}
		+ &-\\-&+
	\end{bmatrix}$. So by step (2)-(II)-(i) of the algorithm $$A_{13}=A_{11}A_{12}A_{23}+A_{12}A_{23}A_{33}=\begin{bmatrix}
		- &+\\+&-
	\end{bmatrix}.$$

	For the block $A_{34}$ apply step (1), take $A_{34}=\begin{bmatrix}
		- \\+
	\end{bmatrix}$. Therefore,
	
	$$A =
	\left[
	\begin{array}{cc|c|cc|c}
		0 & + & + &- & +&  \\
		+ & 0 & - & + & - &  \\
		\hline
		&& 0 & +& -& \\
		\hline
		&&&0&+&-\\
		&&&+&0&+\\
		\hline
		&&&&&-\\
	\end{array} \right].$$

	Since $A_{23}A_{34}=[-]$, so by step (2)-(II)-(i), $A_{24}=A_{23}A_{34}A_{44}=[+]$.
	
	For the block $A_{14}$, since $A_{12}A_{24}+A_{13}A_{34}=\begin{bmatrix}
		+\\-
	\end{bmatrix}$. 
	So by step (2)-(II)-(i), $A_{14}=\begin{bmatrix}
		-\\+
	\end{bmatrix}$.
	
	Therefore,  $$A =
	\left[
	\begin{array}{cc|c|cc|c}
		0 & + & + &- & +& - \\
		+ & 0 & - & + & - &+  \\
		\hline
		&& 0 & +& -& +\\
		\hline
		&&&0&+&-\\
		&&&+&0&+\\
		\hline
		&&&&&-\\
	\end{array} \right],$$

	which is obtained by the algorithm and $A^{2+1}=A$.
	
\end{eg}

Next, we show that any matrix obtained by the above algorithm is a sign $k$-potent sign pattern. The following theorem gives the result when the number of diagonal blocks is equal to 2. 

\begin{theorem} \label{3.6a}
	Let $A=\begin{bmatrix}
		A_{11}&A_{12} \\
		0&A_{22}
	\end{bmatrix}$ be a nonzero sign pattern in cyclic normal form and suppose that $A$ is in reduced form. If $A$ is obtained by the algorithm, then $A$ is a sign $k$-potent sign pattern.
\end{theorem}
\begin{proof}
	Let $A^{k+1}=\begin{bmatrix}
		C_{11}&C_{12} \\
		0&C_{22}
	\end{bmatrix}$. Then $C_{11}=A_{11}^{k+1}=A_{11}$, $C_{22}=A_{22}^{k+1}=A_{22}$ and $C_{12}=\sum_{h=0}^{k}A_{11}^{k-h}A_{12}A_{22}^h$. Since $A$ is a nonzero sign pattern in cyclic normal form, either $A_{11}$ or $A_{22}$ is a nonzero block.
	
	If both $A_{11}$ and $A_{22}$ are nonzero blocks, then since $A_{12}$ is obtained by using Tables (\ref{4.1}), (\ref{4.1.1}), so by Theorem \ref{4.66}, \ref{4.77} and \ref{4.88}, $A_{11}A_{12}=A_{12}A_{22}$. Therefore, $A_{11}^{k-h}A_{12}A_{22}^h=A_{11}^kA_{12}=A_{12}A_{22}^k=A_{12}$ for all $h=0,1,...,k$. So $C_{12}=\sum_{h=0}^{k}A_{11}^{k-h}A_{12}A_{22}^h=A_{12}$.
	
	If $A_{11}\neq {0}$, $A_{22}= {0}$, then $C_{12}=A_{11}^kA_{12}=A_{12}$. Similarly, if $A_{11}= {0}$, $A_{22}\neq {0}$, then $C_{12}=A_{12}$. Therefore, $A$ is a sign $k$-potent sign pattern.
\end{proof}
%%%%%%%%%%%%%%%%%%%%%%%%
The following lemma shows that the sign patterns obtained by the above algorithm are sign $k$-potent when the number of diagonal blocks of $A$ is more than $2$.

\begin{lemma}\label{l3.1}
	Let $A$ be a reducible sign pattern in cyclic normal form with $n$ diagonal blocks, and let $A$ be in reduced form.  %Suppose that all the blocks of the first $j-1$ columns and the blocks $A_{j-1,j},A_{j-2,j},...,A_{i+1,j}, A_{ij}$ of $A$ are obtained according to the algorithm where $1\leq i\leq j-2$, $3\leq j\leq n$. Then the block $A_{ij}$ satisfies the condition of $k$-potence.
	Suppose that all off-diagonal blocks of $A$ are obtained by the algorithm, then all off-diagonal blocks satisfy the condition of $k$-potence.
\end{lemma}

\begin{proof}
	To prove the above lemma, we need to show that $$\sum_{i_1\leq i_2\leq \cdots \leq i_k=i}^{j}A_{ii_1}A_{i_1i_2}\cdots A_{i_kj}=A_{ij},$$
	for $i,j$; $1\leq i<j\leq n$.
	We apply induction on $j$ where $j$ varies from $2$ to $n$. Since $A_{ij}$ is a subpattern of $\sum_{i_1\leq i_2\leq \cdots \leq i_k=i}^{j}A_{ii_1}A_{i_1i_2}\cdots A_{i_kj}$, it is enough to show that $\sum_{i_1\leq i_2\leq \cdots \leq i_k=i}^{j}A_{ii_1}A_{i_1i_2}\cdots A_{i_kj}$ is a subpattern of $A_{ij}$. For $j=2$, $$\sum_{i_1\leq i_2\leq \cdots \leq i_k=i}^{j}A_{ii_1}A_{i_1i_2}\cdots A_{i_kj}=\sum_{p=0}^kA_{11}^{k-p}A_{12}A_{22}^p.$$
	
	\textbf{Case~1:} Both $A_{11}$, $A_{22}$ are zero blocks.\\ Then by step~(I) of the algorithm $A_{12}=0$, hence $$\sum_{p=0}^kA_{11}^{k-p}A_{12}A_{22}^p=A_{12}.$$
	
	\textbf{Case~2:} Exactly one of $A_{11}$, $A_{22}$ is a zero block.\\ Without loss of generality let $A_{11}=0$, then $$\sum_{p=0}^kA_{11}^{k-p}A_{12}A_{22}^p=A_{12}A_{22}^k=A_{12}.$$
	
	\textbf{Case~3:} Both $A_{11}$, $A_{22}$ are nonzero blocks.\\ Since $A_{12}$ is obtained by using Tables (\ref{4.1}), (\ref{4.1.1}), so by Theorem \ref{4.66}, \ref{4.77} and \ref{4.88}, $A_{11}A_{12}=A_{12}A_{22}$. Therefore, $$\sum_{p=0}^kA_{11}^{k-p}A_{12}A_{22}^p=A_{11}^kA_{12}=A_{12}A_{22}^k=A_{12}.$$
	
	Suppose that the statement is true for all $j\leq m-1$. Therefore, by the induction hypothesis, all the blocks in the first $m-1$ columns satisfy the condition of $k$-potence. 
		For $j=m$, if $i=m-1$, then similarly by the previous argument, $A_{m-1,m}$ satisfies the condition of $k$-potence. For any $1\leq i\leq m-2$, suppose that the blocks $A_{m-1,m},A_{m-2,m},...,A_{i+1,m}$ satisfy the condition of $k$-potence. Now, for the block $A_{im}$, we have the following cases.
	
	\textbf{Case~1:} At least one of $A_{ii}$, $A_{mm}$ is a nonzero block.
	\\Without loss of generality, assume that $A_{ii}$ is a nonzero block.
	Then $A_{ii}^k$ is an identity matrix. Consider the nonzero term $A_{ii_1}A_{i_1i_2}\cdots A_{i_km}$, $\{i_1,i_2,...,i_k\}\not\subset\{i,m\}$. 
	
	If $i_k \leq m-1$, then by the induction hypothesis, $A_{ii}A_{ii_1}A_{i_1i_2}\cdots A_{i_{k-1}i_{k}}$ is a subpattern of $A_{ii_k}$. Since, $$A_{ii_1}A_{i_1i_2}\cdots A_{i_km}=A_{ii}^kA_{ii_1}A_{i_1i_2}\cdots A_{i_km}={A_{ii}^{k-1}(A_{ii}A_{ii_1}A_{i_1i_2}\cdots A_{i_{k-1}i_{k}})A_{i_km}}.$$
	Therefore, $A_{ii_1}A_{i_1i_2}\cdots A_{i_km}$ is a subpattern of $A_{ii}^{k-1}A_{ii_k}A_{i_km}$, which is a subpattern of $\sum_{p=1}^{k}A_{ii}^{k-p}X_{im}A_{mm}^{p-1}$.
	
	If $i_k=m$, let $l=\max\{1,2,...,k\}$ be such that $i_l<m$, then $i+1\leq i_l\leq m-1$ and
	$$A_{ii_1}A_{i_1i_2}\cdots A_{i_km}=A_{ii}^kA_{ii_1}A_{i_1i_2}\cdots A_{i_lm}A_{mm}^{k-l}={A_{ii}^{l-1}(A_{ii}^{k+1-l}A_{ii_1}A_{i_1i_2}\cdots A_{i_{l-1}i_{l}}A_{i_lm})A_{mm}^{k-l}},$$
	is a subpattern of $\sum_{p=1}^{k}A_{ii}^{k-p}X_{im}A_{mm}^{p-1}$ by the induction hypothesis. Hence, $\sum_{\substack{i_1\leq i_2\leq \cdots \leq i_k=i\\\{i_1,i_2,...,i_k\}\not\subset\{i,m\}}}^{
		m}A_{ii_1}A_{i_1i_2}\cdots A_{i_km}$ is a subpattern of $\sum_{p=1}^{k}A_{ii}^{k-p}X_{im}A_{mm}^{p-1}$.
	
	Since, according to the algorithm $A_{im}$ is obtained by  step~(2)-(II), so $\sum_{p=1}^{k}A_{ii}^{k-p}X_{im}A_{mm}^{p-1}$ is a subpattern of $A_{im}$. Therefore, $\sum_{i_1\leq i_2\leq \cdots \leq i_k=i}^{m}A_{ii_1}A_{i_1i_2}\cdots A_{i_km}$ is a subpattern of $A_{im}$.

	\textbf{Case~2:} Both $A_{ii}$, $A_{mm}$ are zero blocks.
	\\Since $A_{im}$ is chosen by step~(2)-(I), so $A_{im}=\sum_{r=i+1}^{j-1}A_{ir}A_{rr}^{k-1} A_{rm}$. Also, $$\sum_{i_1\leq i_2\leq \cdots \leq  i_k=i}^{m}A_{ii_1}A_{i_1i_2}\cdots A_{i_km}=\sum_{i_1\leq i_2\leq \cdots \leq i_k=i+1}^{m-1}A_{ii_1}A_{i_1i_2}\cdots A_{i_km}.$$
	Let $A_{ii_1}A_{i_1i_2}\cdots A_{i_km}$ be a nonzero term from $\sum_{i_1\leq i_2\leq \cdots \leq i_k=i+1}^{m-1}A_{ii_1}A_{i_1i_2}\cdots A_{i_km}$. Then we have the following cases.
	
    \textbf{Subcase~1:} $A_{i_1i_1}$ is a nonzero block.
	\\Then, $$A_{ii_1}A_{i_1i_2}\cdots A_{i_km}=A_{ii_1}A_{i_1i_1}^kA_{i_1i_2}\cdots A_{i_km}=A_{ii_1}A_{i_1i_1}^{k-1}(A_{i_1i_1}A_{i_1i_2}\cdots A_{i_km}).$$
	Since $i_1\geq i+1$ and $A_{i_1m}$ satisfy the condition of $k$-potence, so $A_{i_1i_1}A_{i_1i_2}\cdots A_{i_km}$ is a subpattern of $A_{i_1m}$. Therefore, $A_{ii_1}A_{i_1i_2}\cdots A_{i_km}$  is a subpattern of $A_{ii_1}A_{i_1i_1}^{k-1}A_{i_1m}$, which is a subpattern of $A_{im}$, since $A_{im}$ is chosen by step~(2)-(I) of the algorithm.
	
    \textbf{Subcase~2:} $A_{i_1i_1}$ is a zero block.\\
	Since $i_1\leq m-1$ and $A_{ii}=0$, $A_{ii_1}$ is obtained by step~(2)-(I) of the algorithm, therefore $$A_{ii_1}=\sum_{r=i+1}^{i_1-1}A_{ir}A_{rr}^{k-1}A_{ri_1}.$$
	So, $$A_{ii_1}A_{i_1i_2}\cdots A_{i_km}=\sum_{r=i+1}^{i_1-1}A_{ir}A_{rr}^{k-1}(A_{ri_1}A_{i_1i_2}\cdots A_{i_km}).$$
	Since $A_{ii_1}\neq 0$, there exists $r$, $i+1\leq r\leq i_1-1$, such that $A_{rr}\neq 0$. Also, $A_{rm}$ satisfies the condition of $k$-potence, so $A_{ri_1}A_{i_1i_2}\cdots A_{i_km}$ is a subpattern of $A_{rm}$. Hence, $A_{ii_1}A_{i_1i_2}\cdots A_{i_km}$  is a subpattern of $\sum_{r=i+1}^{i_1-1}A_{ir}A_{rr}^{k-1}A_{rm}$, which is a subpattern of $A_{im}$. Therefore, $$\sum_{i_1\leq i_2\leq \cdots \leq i_k=i}^{m}A_{ii_1}A_{i_1i_2}\cdots A_{i_km}=A_{im}.$$
\end{proof}
%%%%%%%%%%%%%%%%%%%%%%%%%%%%%%%%%%%%%%%%%%%%%%%%%%%%%%%%%%%

\begin{theorem}
	Let $A$ be a reducible sign pattern in cyclic normal form and let $A$ be in reduced form. Suppose that $A$ is a sign $k$-potent sign pattern. Then, all off-diagonal blocks of $A$ can be obtained by the algorithm.
\end{theorem}
\begin{proof}
	Let $A_{ij}$ be an off-diagonal block of $A$, then $A_{ij}$ satisfies the conditions of the Tables (\ref{4.1}), (\ref{4.1.1}), and the condition of $k$-potence.
	If $i=j-1$, then $$A_{j-1,j}=\sum_{p=1}^{k}A_{j-1,j-1}^{k-p}A_{j-1,j}A_{jj}^{p-1}.$$ Also,  $$\sum_{p=1}^{k}A_{ll}^{k-p}X_{j-1,j}^lA_{jj}^{p-1}+A_{l,j-1}A_{j-1,j-1}^{k-1}A_{j-1,j},$$ which is a subpattern of $A_{lj}$, is defined unambiguously for all $l=j-2,...,1$. So, the block $A_{j-1,j}$ can be obtained by step~(1) of the algorithm.

	For $i\leq j-2$, since $A_{ij}$ satisfies the condition of $k$-potence,
	$$A_{ij}=\sum_{i_1\leq i_2\leq \cdots \leq i_k=i}^{j}A_{ii_1}A_{i_1i_2}\cdots A_{i_kj}.$$
	If both $A_{ii}$, $A_{jj}$ are zero blocks, then by Lemma \ref{l3.1},
	$A_{ij}$ can be obtained by step~(2)-(I) of the algorithm.
	
	Assume that at least one of $A_{ii}$, $A_{jj}$ is a nonzero block. Since $\sum_{p=1}^{k}A_{ll}^{k-p}X_{ij}^lA_{jj}^{p-1}+\sum_{r=i}^{j-1}A_{lr}A_{rr}^{k-1}A_{rj}$ is a subpattern of $A_{lj}$, $l=i-1,...,1$, $\sum_{p=1}^{k}A_{ll}^{k-p}X_{ij}^lA_{jj}^{p-1}+\sum_{r=i}^{j-1}A_{lr}A_{rr}^{k-1}A_{rj}$ is defined unambiguously. Then, by Lemma \ref{l3.1}, $A_{ij}$ can be obtained by step~(2)-(II) of the algorithm.
\end{proof}
%%%%%%%%%%%%%%%%%%%%%%%%%%%%%%%%%%%%%%%%%%%%%%%%%%%%%%%%

The following example shows that the above algorithm for sign $2$-potent sign patterns does not always terminate in a single iteration.
\begin{eg} Take,
	$$A =
	\left[
	\begin{array}{c|cc|c|c}
		0 & + & - &+&  \\ \hline
		& 0 & + & 0 &    \\
		
		&+& 0 & +& \\
		\hline
		&&&0&\\ \hline
		&&&&+
		
	\end{array} \right],$$ then the blocks in the first three columns of $A$ are obtained by using the algorithm. For the blocks in the fourth superdiagonal if we take $A_{34}=[+]$ by step~(1), then by step~(2)-(II), $$A_{24}=A_{22}X_{24}+X_{24}A_{44}=\begin{bmatrix}
		+\\+
	\end{bmatrix},$$ then it is not possible to construct the block $A_{14}$ in the immediate next step by using this algorithm. 
    %The algorithm will then again start with a different choice of block $A_{12}$ to finally give a sign $k$-potent sign pattern.
\end{eg}

However, suppose the sign pattern $A$ has at most one run of zero diagonal blocks. In that case, that is, at most one set of maximal consecutive zero diagonal blocks, then the above algorithm will terminate in a single iteration, as the following theorem shows.
%%%%%%%%%%%%%%%%%%%%%%%%%%%%%%%%%%%

\begin{theorem} \label{4.3aa}

	Let $A$ be a reducible sign pattern in cyclic normal form with $n$ diagonal blocks and let $A$ be in reduced form. Suppose that $A$ has at most one run of zero diagonal blocks and suppose that the first $j-1$ columns of $A$ for $3\leq j\leq n$ and the blocks $A_{jj},A_{j-1,j}, A_{j-2,j},..., A_{i+1,j}$ where $1\leq i \leq j-1$ are obtained by the algorithm. %If $A_{ii}$ or $A_{jj}$ is a nonzero block then,
	%	$A_{ii}(\sum_{p=i+1}^{j-1}A_{ip}A_{pj})+(\sum_{p=i+1}^{j-1}A_{ip}A_{pj})A_{jj}$ is subpattern of a sign pattern matrix whose structure
	Then $A_{ij}$ can be obtained by the algorithm in the next step. % for $i=j-2,...,1$.
\end{theorem} 

\begin{proof}	If $i=j-1$, then $A_{j-1,j}$ can be obtained by step~(1) of the algorithm by choosing $A_{j-1,j}=0$ or a superpattern of zero.
	
	%For $1\leq i\leq j-2$,  if both $A_{ii}$, $A_{jj}$ are zero blocks, since $A$ has at most one run of zero blocks $\sum_{r=i+1}^{j-1}A_{ir}A_{rr}A_{rj}=0$, so $A_{ij}$ can be obtained by step~(2)-(I) of the algorithm. 
	
	%  If at least one of $A_{ii},$ $A_{jj}$ is a nonzero block then we prove the above theorem in two parts.	
	
	For $1\leq i\leq j-2$, we prove the above theorem in two parts. 

    \textbf{Part}-1: We show that $$\sum_{p=1}^{k}A_{ii}^{k-p}X_{ij}A_{jj}^{p-1}$$ is defined unambiguously, where $X_{ij}=\sum_{p=i+1}^{j-1}A_{ip}A_{pj}$ and 
	$$\sum_{p=1}^{k}A_{ll}^{k-p}X_{ij}^lA_{jj}^{p-1}+\sum_{r=i}^{j-1}A_{lr}A_{rr}^{k-1}A_{rj}$$
	is defined unambiguously, where $A_{ij}=\sum_{p=1}^{k}A_{ii}^{k-p}X_{ij}A_{jj}^{p-1}$  and $X_{ij}^l=\sum_{p=i}^{j-1}A_{lp}A_{pj}$ for $1\leq l\leq i-1$.

\textbf{Part}-2: We next show that $\sum_{p=1}^{k}A_{ii}^{k-p}X_{ij}A_{jj}^{p-1}$ satisfies the conditions of Tables (\ref{4.1}), (\ref{4.1.1}).

If both $A_{ii}$, $A_{jj}$ are zero blocks, since $A$ has at most one run of zero blocks, $\sum_{r=i+1}^{j-1}A_{ir}A_{rr}^{k-1}A_{rj}=0$, so $A_{ij}$ can be obtained by step~(2)-(I) of the algorithm. 
If at least one of $A_{ii},$ $A_{jj}$ is a nonzero block, then
by combining the above two parts $A_{ij}$ in this case can be obtained by step~(2)-(II) of the algorithm.
	
    \textbf{Proof of Part-1.}
     We apply induction on $i$, where $i$ varies from $j-2$ to $1$. For, $i=j-2$ $$\sum_{p=1}^{k}A_{ii}^{k-p}X_{ij}A_{jj}^{p-1}=\sum_{p=1}^{k}A_{j-2,j-2}^{k-p}X_{j-1,j}^{j-2}A_{jj}^{p-1}=\sum_{p=1}^{k}A_{ll}^{k-p}X_{j-1,j}^lA_{jj}^{p-1}$$
	for $l=j-2$ is defined unambiguously, since $A_{j-1,j}$ is chosen by step~(1) of the algorithm.
	%Without loss of generality suppose that $A_{jj}$ is a nonzero block. If $A_{j-1,j-1}=0$, then
	%{$$\sum_{r=j-1}^{j-1}A_{j-2,r}A_{rr}A_{rj}=A_{j-2,j-1}A_{j-1,j-1}A_{j-1,j}$$} is defined unambiguously. Suppose $A_{j-1,j-1}\neq 0$, since $A_{j-1,j}$ is obtained according to the table (\ref{4.1}), so by Theorems \ref{4.66}, \ref{4.77} and \ref{4.88}, 
	%$$A_{j-2,j-1}A_{j-1,j-1}A_{j-1,j}=A_{j-2,j-1}A_{j-1,j}A_{jj}.$$ %Hence, $\sum_{r=j-1}^{j-1}A_{j-2,r}A_{rr}A_{rj}$ is defined unambiguously.
	
	Again, for any $l$, $1\leq l\leq j-3$,
	$$\sum_{p=1}^{k}A_{ll}^{k-p}X_{j-2,j}^lA_{jj}^{p-1}+\sum_{r=j-2}^{j-1}A_{lr}A_{rr}^{k-1}A_{rj}=\sum_{p=1}^{k}A_{ll}^{k-p}X_{j-1,j}^lA_{jj}^{p-1}+\sum_{r=j-1}^{j-1}A_{lr}A_{rr}^{k-1}A_{rj}$$$$+\sum_{p=1}^{k}A_{ll}^{k-p}A_{l,j-2}A_{j-2,j}A_{jj}^{p-1}+{A_{l,j-2}A_{j-2,j-2}^{k-1}A_{j-2,j}},$$
	%$$\sum_{r=j-2}^{j-1}A_{lr}A_{rr}A_{rj}=\sum_{r=j-1}^{j-1}A_{lr}A_{rr}A_{rj}+A_{l,j-2}A_{j-2,j-2}A_{j-2,j}.$$
	where, $\sum_{p=1}^{k}A_{ll}^{k-p}X_{j-1,j}^lA_{jj}^{p-1}+\sum_{r=j-1}^{j-1}A_{lr}A_{rr}^{k-1}A_{rj}$ is defined unambiguously,  since $A_{j-1,j}$ is chosen by step~(1) of the algorithm.

	Suppose that $A_{ll}^{k-p}A_{l,j-2}A_{j-2,j}A_{jj}^{p-1}\neq 0$ for some $p$, $1\leq p\leq k$. Note that if $$A_{j-2,j}=
    \sum_{r=1}^{k}A_{j-2,j-2}^{k-r}A_{j-2,j-1}A_{j-1,j}A_{jj}^{r-1},$$ then
	$$A_{ll}^{k-p}A_{l,j-2}A_{j-2,j}A_{jj}^{p-1}=\sum_{r=1}^{k}A_{ll}^{k-p}(A_{l,j-2}A_{j-2,j-2}^{k-r}A_{j-2,j-1}A_{j-1,j})A_{jj}^{p+r-2}.$$

	Since $A_{j-2,j-1}\neq 0$ then either $A_{j-2,j-2}\neq 0$ or $A_{j-1,j-1}\neq 0$. Now $A$ has at most one run of zero diagonal blocks, so $A_{ll}\neq 0$ or $A_{j-1,j-1}\neq 0$. Without loss of generality, suppose that $A_{ll}\neq0$.
    
    %If  $p+r\leq k+1$ for any $r$, $1\leq r\leq k$,
	%$$A_{ll}^{k-p}A_{l,j-2}A_{j-2,j-2}^{k-r}A_{j-2,j-1}A_{j-1,j}A_{jj}^{p+r-2}=A_{ll}^{k-(p+r-1)}(A_{ll}^{r-1}A_{l,j-2}A_{j-2,j-2}^{k-r}A_{j-2,j-1})A_{j-1,j}A_{jj}^{p+r-2}.$$ 
    
    Since $A_{l,j-1}$ is obtained by the algorithm, by Lemma \ref{l3.1}, $A_{ll}^{r-1}A_{l,j-2}A_{j-2,j-2}^{k-r}A_{j-2,j-1}$ is a subpattern of $A_{l,j-1}$, so $A_{ll}^{k-p}A_{l,j-2}A_{j-2,j}A_{jj}^{p-1}$ is a subpattern of $\sum_{s=1}^{k}A_{ll}^{k-s}X_{j-1,j}^lA_{jj}^{s-1}$. 
	
	%If $p+r>k+1$ for any $r$, $1\leq r\leq k$,
	%\begin{equation*}\begin{split}
	%		A_{ll}^{k-p}A_{l,j-2}A_{j-2,j-2}^{k-r}A_{j-2,j-1}A_{j-1,j}A_{jj}^{p+r-2} &= A_{ll}^{k}(A_{ll}^{k-p}A_{l,j-2}A_{j-2,j-2}^{k-r}A_{j-2,j-1}A_{j-1,j}A_{jj}^{p+r-2})\\
	%		&=  A_{ll}^{2k-p-r+1}(A_{ll}^{r-1}A_{l,j-2}A_{j-2,j-2}^{k-r}A_{j-2,j-1})A_{j-1,j}A_{jj}^{p+r-k-2}A_{jj}^{k}\\ 
	%%\end{split}	\end{equation*}
	%By the previous argument, $A_{ll}^{k-p}A_{l,j-2}A_{j-2,j}A_{jj}^{p-1}$ is a subpattern of $\sum_{s=1}^{k}A_{ll}^{k-s}X_{j-1,j}^lA_{jj}^{s-1}$.

	Similarly, ${A_{l,j-2}A_{j-2,j-2}^{k-1}A_{j-2,j}}$ is a subpattern of $\sum_{p=1}^{k}A_{ll}^{k-p}X_{j-1,j}^lA_{jj}^{p-1}+\sum_{r=j-1}^{j-1}A_{lr}A_{rr}^{k-1}A_{rj}$.
	
	% Therefore, $A_{j-2,j}$ can be chosen by step~(2)-(II) of the algorithm.
	
	Suppose that the statement is true for all $i\geq m+1$. Then for $i=m+1$, $A_{m+1,j}$ is chosen by the algorithm, so by the induction hypothesis %that $A_{k+1,k+1}X_{k+1,j}+X_{k+1,j}A_{jj}$ is a subpattern of $A_{k+1,j}$ and 
	\begin{equation}\label{con3}
		\sum_{p=1}^{k}A_{ll}^{k-p}X_{m+1,j}^lA_{jj}^{p-1}+\sum_{r=m+1}^{j-1}A_{lr}A_{rr}^{k-1}A_{rj}
	\end{equation}
	is defined unambiguously for $1\leq l\leq m$.
	
	For $i=m$, by condition (\ref{con3}), $$\sum_{p=1}^{k}A_{mm}^{k-p}X_{mj}A_{jj}^{p-1}=\sum_{p=1}^{k}A_{mm}^{k-p}X_{m+1,j}^mA_{jj}^{p-1}=\sum_{p=1}^{k}A_{ll}^{k-p}X_{m+1,j}^lA_{jj}^{p-1},$$ for $l=m$ is defined unambiguously. Now for any $l$, $1\leq l\leq m-1$,

\begin{align}\label{con5}
    \sum_{p=1}^{k}A_{ll}^{k-p}X_{mj}^lA_{jj}^{p-1}+\sum_{r=m}^{j-1}A_{lr}A_{rr}^{k-1}A_{rj} &= \bigg(\sum_{p=1}^{k}A_{ll}^{k-p}X_{m+1,j}^lA_{jj}^{p-1}+\sum_{r=m+1}^{j-1}A_{lr}A_{rr}^{k-1}A_{rj}\bigg) \notag \\
      &\quad +\sum_{p=1}^{k}A_{ll}^{k-p}A_{lm}A_{mj}A_{jj}^{p-1}+A_{lm}A_{mm}^{k-1}A_{mj}.
\end{align}

 Since condition (\ref{con3}) is satisfied, showing that (\ref{con5}) is defined unambiguously it is enough to show that $\sum_{p=1}^{k}A_{ll}^{k-p}A_{lm}A_{mj}A_{jj}^{p-1}+A_{lm}A_{mm}^{k-1}A_{mj}$ is a subpattern of $\sum_{p=1}^{k}A_{ll}^{k-p}X_{m+1,j}^lA_{jj}^{p-1}+\sum_{r=m+1}^{j-1}A_{lr}A_{rr}^{k-1}A_{rj}$.
	
	Suppose that  $A_{ll}^{k-p}A_{lm}A_{mj}A_{jj}^{p-1}\neq 0$ for some $p$, $1\leq p\leq k$. Since $A_{mj}=\sum_{r=1}^{k}A_{mm}^{k-r}X_{mj}A_{jj}^{r-1}$ so,
	$$A_{ll}^{k-p}A_{lm}A_{mj}A_{jj}^{p-1}=\sum_{r=1}^{k}\sum_{s=m+1}^{j-1}A_{ll}^{k-p}(A_{lm}A_{mm}^{k-r}A_{ms})A_{sj}A_{jj}^{p+r-2}.
	$$ 	Suppose that $A_{ms}\neq 0$ for some $s$, $m+1\leq s\leq j-1$, so there exists $s_1$, $m\leq s_1\leq s$ such that $A_{s_1s_1}\neq 0$. Since $A$ has at most one run of zero diagonal blocks, either $A_{ll}\neq 0$ or $A_{ss}\neq 0$. So, without loss of generality, let $A_{ll}\neq 0$.
    
    It can be shown that $A_{ll}^{r-1}A_{lm}A_{mm}^{k-r}A_{ms}$ is a subpattern of $A_{ls}$, which implies that $A_{ll}^{k-p}A_{lm}A_{mj}A_{jj}^{p-1}$ is a subpattern of $ A_{ll}^{k-p_1}A_{ls}A_{sj}A_{jj}^{p_1-1}$ for some $p_1$. Hence, $A_{ll}^{k-p}A_{lm}A_{mj}A_{jj}^{p-1}$ is a subpattern of $\sum_{p=1}^{k}A_{ll}^{k-p}X_{m+1,j}^lA_{jj}^{p-1}$.
	
    %If $p+r\leq k+1$ for some $r$, $1\leq r \leq k$, then
%	$$A_{ll}^{k-p}A_{lm}A_{mm}^{k-r}A_{ms}A_{sj}A_{jj}^{p+r-2}=A_{ll}^{k-(p+r-1)}(A_{ll}^{r-1}A_{lm}A_{mm}^{k-r}A_{ms})A_{sj}A_{jj}^{p+r-2}$$ is a subpattern of $ A_{ll}^{k-(p+r-1)}A_{ls}A_{sj}A_{jj}^{p+r-2}$. Hence, $A_{ll}^{k-p}A_{lm}A_{mj}A_{jj}^{p-1}$ is a subpattern of $\sum_{p=1}^{k}A_{ll}^{k-p}X_{m+1,j}^lA_{jj}^{p-1}$. 
	
%	If $p+r> k+1$ for some $r$, $1\leq r \leq k$, then
%	\begin{equation*}
%		\begin{split}
%			A_{ll}^{k-p}A_{lm}A_{mm}^{k-r}A_{ms}A_{sj}A_{jj}^{p+r-2}&=A_{ll}^k(A_{ll}^{k-p}A_{lm}A_{mm}^{k-r}A_{ms}A_{sj}A_{jj}^{p+r-2})\\
%			&=A_{ll}^{2k-p-r+1}(A_{ll}^{r-1}A_{lm}A_{mm}^{k-r}A_{ms})A_{sj}A_{jj}^{p+r-k-2}A_{jj}^k\\
%			&=A_{ll}^{2k-p-r+1}(A_{ll}^{r-1}A_{lm}A_{mm}^{k-r}A_{ms})A_{sj}A_{jj}^{p+r-k-2}
%		\end{split}
%	\end{equation*}

	Similarly, $A_{lm}A_{mm}^{k-1}A_{mj}$ is a subpattern of $\sum_{p=1}^{k}A_{ll}^{k-p}X_{m+1,j}^lA_{jj}^{p-1}+\sum_{r=m+1}^{j-1}A_{lr}A_{rr}^{k-1}A_{rj}$.

	%Therefore, $A_{kj}$ can be chosen by step~(2)-(II) of the algorithm.
    
	%%%%%%%%%%%%%%%%%%%%%%%%
	
	\textbf{Proof of Part-2.} If exactly one of $A_{ii}$, $A_{jj}$ is a zero block, then $\sum_{p=1}^{k}A_{ii}^{k-p}X_{ij}A_{jj}^{p-1}$ is a sign pattern row or column vector, and hence it satisfies the conditions of Tables (\ref{4.1}), (\ref{4.1.1}).
	So, let both $A_{ii}$ and $A_{jj}$ be nonzero diagonal blocks. 	Let $$\sum_{p=1}^{k}A_{ii}^{k-p}X_{ij}A_{jj}^{p-1}=Z.$$
	By Theorems \ref{4.66}, \ref{4.77} and \ref{4.88} to prove this results it is enough to show that $A_{ii}Z=ZA_{jj}$.
	Let $A_{ii}^{k-p}A_{is}A_{sj}A_{jj}^{p-1}$ be an arbitrary term of $Z$ where $1\leq p\leq k$, $i+1\leq s \leq j-1$.
	%Case-1: $1<p\leq k$.
	Then $$A_{ii}(A_{ii}^{k-p}A_{is}A_{sj}A_{jj}^{p-1})=(A_{ii}^{k-p+1}A_{is}A_{sj}A_{jj}^{p-2})A_{jj}.$$
	Since $A_{ii}^{k-p+1}A_{is}A_{sj}A_{jj}^{p-2}$ is a subpattern of $Z$ therefore, $A_{ii}Z$ is a subpattern of $ZA_{jj}$.
%Case-2: $p=1$.
	%Since $A_{jj}$ is a nonzero block, %Thus,  $$A_{ii}(A_{ii}^{k-1}A_{is}A_{sj})=A_{ii}^{k}A_{is}A_{sj}A_{jj}^k=(A_{is}A_{sj}A_{jj}^{k-1})A_{jj}.$$
	%Since $A_{is}A_{sj}A_{jj}^{k-1}$ is a subpattern of $Z$ therefore, $A_{ii}Z$ is a subpattern of $ZA_{jj}$. 
    %Therefore, $A_{ii}Z$ is a subpattern of $ZA_{jj}$.
	Similarly, $ZA_{jj}$ is a subpattern of $A_{ii}Z$. Thus $A_{ii}Z=ZA_{jj}$.
\end{proof}
%%%%%%%%%%%%%%%%%%%%%%%%%%%%%%%%%%%%%%%%%%%%%%%%%%%%%%%%%%%%%%%%

\section{Sign $k$-potent sign patterns that allow $k$-potence} \label{s5}
In this section, we find some conditions for a sign $k$-potent sign pattern to allow $k$-potence.
Let $A$ be an irreducible sign $k$-potent sign pattern. Stuart et al. \cite{05} proved that $A$ is signature and permutation similar to
\begin{equation}\label{1.3aa}
	\begin{bmatrix}
		0&J_{n_1\times n_2}&0&0&\cdots&0 \\
		0&0&J_{n_2\times n_3}&0&\cdots&0 \\
		\vdots&&&\ddots&&\vdots \\
		0&0&0&0&\cdots&J_{n_{m-1}\times n_m}\\
		\alpha J_{n_m\times n_1}&0&0&0&\cdots&0
	\end{bmatrix}
\end{equation}

where each $J_{n_i\times n_{i+1}}$, $1\leq i\leq m-1$ and $J_{n_m\times n_1}$ is entry wise positive sign pattern, $\alpha \in \{+,-\}$.\\

If $A=(a_{ij})$ is in the above form, take the block matrix $B=(B_{ij})$ such that

\begin{equation}\label{1.4aa}
	B_{ij} = \begin{cases} 
		[\frac{1}{n_{i}}] & if ~ A_{i,i+1}=J_{n_i\times n_{i+1}},~ i=1,\cdots,m-1, \\
		[\frac{\beta}{n_m}] & if ~A_{m, 1}=\alpha J_{n_m\times n_1}, \\
		0 & otherwise.
	\end{cases}
\end{equation}
Where $\beta \in \{0,1,-1\}$ with $\sgn(\beta)=\alpha$, then $B$ is a $k$-potent matrix \cite{04}. 
Thus, every irreducible sign $k$-potent sign pattern allows $k$-potence. Therefore, we consider $A$ to be a reducible sign $k$-potent sign pattern in cyclic normal form.

\begin{theorem} \label{5.1aa}
	Let $A=\begin{bmatrix}
		A_{11}&A_{12} \\
		0&A_{22}
	\end{bmatrix}$ be a reducible sign $k$-potent sign pattern in cyclic normal form. Then $A$ allows $k$-potence if and only if at least one of $A_{11}$, $A_{12}$, $A_{22}$ is a zero block. 
\end{theorem}	
\begin{proof}
	Assume that $A$ allows $k$-potence and $B\in \mathcal{Q}(A)$ is such that $B^{k+1}=B.$
	If at least one of $A_{11}$ or $A_{22}$ is a zero block, then the above theorem holds good.
	
	Let both $A_{11}$, $A_{22}$ be nonzero blocks. Since $A^{k+1}=A$, hence the $(i,j)$-th entry of $A^{k+1}$,  $(A^{k+1})_{ij}=\sum_{i_1,...,i_k=i}^j{a_{ii_1}a_{i_1i_2}\cdots a_{i_kj}}$ is defined unambiguously.
	Thus $|A|^{k+1}=|A|$, where $|A|=[|a_{ij}|]$. 
	So, $B\in \mathcal{Q}(A)$ and $B^{k+1}=B$ implies $|B|=|B|^{k+1}$.
	
	Thus,
	$$|B|_{12}=|B|_{11}^k|B|_{12}+|B|_{12}|B|_{22}^k+\sum_{h=1}^{k-1}|B|_{11}^{k-h}|B|_{12}|B|_{22}^h$$
	$$ \Rightarrow|B|_{11}|B|_{12}|B|_{22}=|B|_{11}^{k+1}|B|_{12}|B|_{22}+|B|_{11}|B|_{12}|B|_{22}^{k+1}+\sum_{h=1}^{k-1}|B|_{11}^{k-h+1}|B|_{12}|B|_{22}^{h+1}$$ 
	$$ \Rightarrow|B|_{11}|B|_{12}|B|_{22}=|B|_{11}|B|_{12}|B|_{22}+|B|_{11}|B|_{12}|B|_{22}+\sum_{h=1}^{k-1}|B|_{11}^{k-h+1}|B|_{12}|B|_{22}^{h+1}~~~~$$
	$$\Rightarrow|B|_{11}|B|_{12}|B|_{22}=0.~~~~~~~~~~~~~~~~~~~~~~~~~~~~~~~~~~~~~~~~~~~~~~~~~~~~~~~~~~~~~~~~~~~~~~~~~~~~~~$$
	
	Since $|B|_{11}$, $|B|_{22}$ are in cyclic normal form, the above implies $|B|_{12}=0$, $A_{12}=0$.
    
	Conversely, suppose that at least one of $(A_{11})_{p_1\times p_1}$, $(A_{12})_{p_1\times p_2}$, and $(A_{22})_{p_2\times p_2}$ is a zero block. 
    
	\textbf{Case~1:} $A_{11}$ and $A_{22}$ are zero blocks.
	
	Since $A$ is sign $k$-potent, $A_{12}$ is also a zero block. So $A=0$, and hence $A$ allow $k$-potence.
    
	\textbf{Case~2:} $A_{11}$ and $A_{22}$ are nonzero blocks, and hence $A_{12}$ is a zero block.
	
	Since both $A_{11}$ and $A_{22}$ are irreducible and in cyclic normal form, therefore they have the form (\ref{1.3aa}). If $B$ is such that $B_{11}$ and $B_{22}$ are taken according to (\ref{1.4aa}) and $B_{12}=0$, then $B\in \mathcal{Q}(A)$ and $B$ is $k$-potent.
    
	\textbf{Case~3:} $A_{11}$, $A_{12}$ are nonzero blocks and $A_{22}$ is a zero block.
	
	Then $A_{11}$ is irreducible, sign $k$-potent and has the form (\ref{1.3aa}). Now take the subblock partition of $A_{12}$ corresponding to the cyclic normal form of $A_{11}$. Then according to Lemma~\ref{11} the $i$-th subblock of $A_{12}$ is of the form $\alpha_{i}J_{n_i\times {1}}$, $i=1,2,...,m$, where $m$ is the number of diagonal blocks in the cyclic normal form of $A_{11}$. 
	
	Take a real matrix $B$ such that $B_{11}$ is taken according to (\ref{1.4aa}) and $B_{12}$ is such that the $i$-th subblock of $B_{12}$ is of the form $[\frac{\beta_{i}}{n_i}]_{n_i\times 1}$, $i=1,2,...,m$ where $\beta_{i}\in \{0,1,-1\}$ and $\sgn(\beta_{i})=\alpha_{i}$. Then $B\in \mathcal{Q}(A)$, $(B^{k+1})_{11}=B_{11}$, $(B^{k+1})_{12}=B_{12}$ and $B$ is $k$-potent. 
	
	Similarly, $A$ allows $k$-potence if both $A_{12}$ and $A_{22}$ are nonzero blocks and $A_{11}$ is a zero block.
\end{proof}

\begin{theorem} \label{5.2}
	Let $A$ be a reducible sign $k$-potent sign pattern in cyclic normal form and suppose that $A$ allows $k$-potence. For positive integers $i,j$ with $i+1<j$, if either $A_{ii}$ or $A_{jj}$ is a nonzero block of $A$, then
	\begin{equation} \label{2.1}
		\sum_{\substack{i_1\leq i_2\leq \cdots \leq i_k=i\\\{i_1,i_2,...,i_k\}\not\subset\{i,j\}}}^{j}(A_{ii_1}A_{i_1i_2}\cdots A_{i_kj})=0.
	\end{equation}
	Moreover, if both $A_{ii}$, $A_{jj}$ are nonzero blocks, then $A_{ij}=0$.
\end{theorem}
\begin{proof}
	Without loss of generality, suppose that $A_{ii}$ is a nonzero block. Since $A^{k+1}=A$ and the $(i,j)$-th entry is $(A^{k+1})_{ij}=\sum_{i_1,...,i_k}{a_{ii_1}a_{i_1i_2}\cdots a_{i_kj}}$ is defined unambiguously. Thus $|A|^{k+1}=|A|$, where $|A|=[|a_{ij|}]$.
	
	So, $B\in \mathcal{Q}(A)$ and $B^{k+1}=B$ implies $|B|=|B|^{k+1}$. Therefore, for $i,j$ with $i+1< j$,
	
	$$|B_{ij}| =\sum_{i_1\leq i_2\leq \cdots \leq i_k=i}^{j}(|B|_{ii_1}|B|_{i_1i_2}\cdots |B|_{i_kj})~~~~~~~~~~~~~~~~~~~~~~ $$
	$$\Rightarrow |B|_{ii}|B|_{ij} =|B|_{ii}^{k+1}|B|_{ij}+|B|_{ii}\left[\sum_{\substack{i_1\leq i_2\leq \cdots \leq i_k=i\\i_k\neq i}}^{j}(|B|_{ii_1}|B|_{i_1i_2}\cdots |B|_{i_kj})\right]$$
	\begin{equation} \label{2.2} \Rightarrow 0 = |B|_{ii}\left[\sum_{\substack{i_1\leq i_2\leq \cdots \leq i_k=i\\i_k\neq i}}^{j}(|B|_{ii_1}|B|_{i_1i_2}\cdots |B|_{i_kj})\right]. ~~~~~~~~~~~~~~~~~~~~\end{equation}

	According to Lemma \ref{11}, if we consider the subblock partition of $A$ induced by the cyclic normal form of the irreducible diagonal block $A_{ii}$ of $A$, then each subblock of $A_{ij}$ is of the form $\alpha_{st}J_{n_s\times n_t}$ for some $s,t$, $\alpha_{st}\in \{0,+,-\}$.
	
	Also, $B_{ii}\in \mathcal{Q}(A_{ii})$ where $A_{ii}$ is in cyclic form. Therefore,
	$$|B|_{ii}\left[\sum_{\substack{i_1\leq i_2\leq \cdots \leq i_k=i\\i_k\neq i}}^{j}(|B|_{ii_1}|B|_{i_1i_2}\cdots |B|_{i_kj})\right]=0 \Rightarrow \sum_{\substack{i_1\leq i_2\leq \cdots \leq i_k=i\\i_k\neq i}}^{j}(|B|_{ii_1}|B|_{i_1i_2}\cdots |B|_{i_kj})=0$$
	Hence,
	$$\sum_{\substack{i_1\leq i_2\leq \cdots \leq i_k=i\\i_k\neq i}}^{j}(B_{ii_1}B_{i_1i_2}\cdots B_{i_kj})=0$$
	\begin{equation} \label{5.5}
		\implies	\sum_{\substack{i_1\leq i_2\leq \cdots \leq i_k=i\\\{i_1,i_2,...,i_k\}\not\subset\{i,j\}}}^{j}(B_{ii_1}B_{i_1i_2}\cdots B_{i_kj})=-\sum_{p=1}^{k}B_{ii}^{k-p}B_{ij}B_{jj}^p.
	\end{equation}
	If $B_{jj}=0$, then $	\sum_{\substack{i_1\leq i_2\leq \cdots \leq i_k=i\\\{i_1,i_2,...,i_k\}\not\subset\{i,j\}}}^{j}(B_{ii_1}B_{i_1i_2}\cdots B_{i_kj})=0$.\\
	If $B_{jj}$ is a nonzero block, by Lemma \ref{4.55}, we have $A_{ii}A_{ij}=A_{ij}A_{jj}$ hence $-A_{ii}^{k-p}A_{ij}A_{jj}^p=-A_{ii}^kA_{ij}=-A_{ij}$, so by (\ref{5.5}),
	$$\sum_{\substack{i_1\leq i_2\leq \cdots \leq i_k=i\\\{i_1,i_2,...,i_k\}\not\subset\{i,j\}}}^{j}(A_{ii_1}A_{i_1i_2}\cdots A_{i_kj})=-\sum_{p=1}^{k}A_{ii}^{k-p}A_{ij}A_{jj}^p=-A_{ij}.$$
	
	Also, $\sum_{\substack{i_1\leq i_2\leq \cdots \leq i_k=i\\\{i_1,i_2,...,i_k\}\not\subset\{i,j\}}}^{j}(A_{ii_1}A_{i_1i_2}\cdots A_{i_kj})$ is a subpattern of $A_{ij}$. 
	Therefore,
	$$\sum_{\substack{i_1\leq i_2\leq \cdots \leq i_k=i\\\{i_1,i_2,...,i_k\}\not\subset\{i,j\}}}^{j}(A_{ii_1}A_{i_1i_2}\cdots A_{i_kj})=0.$$

	Suppose that there are $i,j\in\{1,2,...,m\}$ such that $A_{ii}$ and $A_{jj}$ are nonzero blocks. 
	Since $A$ allows $k$-potence, there exists $B\in \mathcal{Q}(A)$ such that $B^{k+1}=B$, which implies $|B|=|B|^{k+1}$.
	Thus, $$|B|_{ij}=|B|_{ii}^k|B|_{ij}+|B_{ii}|^{k-1}|B_{ij}||B_{jj}|+\cdots+|B_{ii}||B_{ij}||B_{jj}|^{k-1}+|B|_{ij}|B|_{jj}^k$$
	$$ \Rightarrow|B|_{ii}|B|_{ij}|B|_{jj}=|B|_{ii}^{k+1}|B|_{ij}|B|_{jj}+|B_{ii}|^{k}|B_{ij}||B_{jj}|^2+\cdots+|B_{ii}|^2|B_{ij}||B_{jj}|^{k}+|B|_{ii}|B|_{ij}|B|_{jj}^{k+1}$$ 
	$$ \Rightarrow|B|_{ii}|B|_{ij}|B|_{jj}=|B|_{ii}|B|_{ij}|B|_{jj}+|B_{ii}|^{k}|B_{ij}||B_{jj}|^2+\cdots+|B_{ii}|^2|B_{ij}||B_{jj}|^{k}+|B|_{ii}|B|_{ij}|B|_{jj}~~~~$$
	$$ \Rightarrow|B|_{ii}|B|_{ij}|B|_{jj}=0.~~~~~~~~~~~~~~~~~~~~~~~~~~~~~~~~~~~~~~~~~~~~~~~~~~~~~~~~~~~~~~~~~~~~~~~~~~~~~~~~~~~~~~~~~~~~~~$$	Since $|B|_{ii}$, $|B|_{jj}$ are nonzero blocks in cyclic normal form, the above implies $|B|_{ij}=0$, $A_{ij}=0$.
\end{proof}
The next corollary follows directly from Theorem \ref{5.1aa} and Theorem \ref{5.2}.
\begin{cor}
	Let $A$ be a reducible sign $k$-potent sign pattern in cyclic normal form. Suppose that all diagonal blocks of $A$ are nonzero. Then $A$ allows $k$-potence if and only if $A$ is a direct sum of irreducible sign $k$-potent sign patterns.
\end{cor}
The following definition is given by  Lee and Park \cite{04}.
\begin{defn}
	The class of all reducible sign $k$-potent sign patterns in which $A_{ij}=0$ whenever $A_{ii}$, $A_{jj}$ are nonzero blocks is defined as \textbf{PPO}.
\end{defn} 
The following theorem is an extension of the corresponding Theorem 2.9 \cite{04} for sign idempotent sign patterns.

\begin{theorem} \label{2.3}
	Let $A$ be a reducible sign $k$-potent sign pattern in cyclic normal form with $n$ diagonal blocks, $n\geq 3$. If $A$ is in \textbf{PPO} such that at least one of $A_{ii}$, $A_{jj}$ is a nonzero block of $A$ then,
	\begin{equation} \label{5.6}
		\sum_{\substack{i_1\leq i_2\leq \cdots \leq i_k=i\\\{i_1,i_2,...,i_k\}\not\subset\{i,j\}}}^{j}(A_{ii_1}A_{i_1i_2}\cdots A_{i_kj})=0,
	\end{equation}  where $i,j$ are positive integers with $i+1<j$. 
\end{theorem}
\begin{proof}
	Suppose that $A$ is a sign $k$-potent sign pattern which is in \textbf{PPO}. Let there exist $i,j\in \{1,2,...,n\}$, $i+1<j$, such that either $A_{ii}$ or $A_{jj}$ is a nonzero block and equation (\ref{5.6}) is not satisfied. Without loss of generality, let $A_{ii}$ be a nonzero block.
	Since $$ 	\sum_{\substack{i_1\leq i_2\leq \cdots \leq i_k=i\\\{i_1,i_2,...,i_k\}\not\subset\{i,j\}}}^{j}(A_{ii_1}A_{i_1i_2}\cdots A_{i_kj}) \neq 0,$$
	there exists $i_1,i_2,...,i_k$; $\{i_1,i_2,...,i_k\}\not\subset\{i,j\}$ such that $A_{ii_1}A_{i_1i_2}\cdots A_{i_kj}\neq 0$. Let $s=\min\{1,2,...,k\}$ be such that $i_s\neq i$, ($i_0=i$).
	
	Since $A$ is in \textbf{PPO}, $A_{i_si_s}$ is a zero block.
	Note that $A_{i_{s+1}i_{s+1}}$ is also a zero block.
	Because if $A_{i_{s+1}i_{s+1}}$ is a nonzero block, then $A_{ii_{s+1}}(=A_{i_{s-1}i_{s+1}})$ is a zero block since $A_{ii}=A_{i_{s-1}i_{s-1}}$ is a nonzero block and $A$ is in $\textbf{PPO}$. Also, $A$ is sign $k$-potent, so $A_{ii}^{k-1}A_{ii_{s}}A_{i_{s}i_{s+1}}$ is a subpattern of $A_{ii_{s+1}}$. Therefore $A_{ii}^{k-1}A_{ii_{s}}A_{i_{s}i_{s+1}}=0$ which implies $ A_{ii_{s}}A_{i_{s}i_{s+1}}=0$ (since $A$ is sign $k$-potent and $A_{ii}$ is a nonzero block in cyclic normal form). So, $A_{ii_1}A_{i_1i_2}\cdots A_{i_kj}=A_{ii}^{s-1}A_{ii_s}A_{i_si_{s+1}}\cdots A_{i_kj}= 0$ which is a contradiction. Therefore, $A_{i_{s+1}i_{s+1}}$ is a zero block. Consider $i_1,i_2,...,i_{s-1},i_{s+1},...,i_k$ as defined before, vary $i_s$ and let $$l=\max\{i_s~|~A_{ii}^{s-1}A_{ii_s}A_{i_s,i_{s+1}}\cdots A_{i_kj}\neq0\}.$$
	By the previous argument, $A_{ll}$ is a zero block.
Also $$A_{li_{s+1}}=\sum_{t_1\leq t_2 \leq \cdots \leq t_k=l}^{i_{s+1}}(A_{lt_1}A_{t_1t_2}\cdots A_{t_ki_{s+1}})=\sum_{t_1\leq t_2\leq \cdots \leq t_k=l+1}^{i_{s+1}-1}(A_{lt_1}A_{t_1t_2}\cdots A_{t_ki_{s+1}}).$$
	Therefore, 
	$$A_{ii}^{s-1} A_{il}A_{li_{s+1}}\cdots A_{i_kj}
	=A_{ii}^{s-1} A_{il}\left[\sum_{t_1\leq t_2\leq \cdots \leq t_k=l+1}^{i_{s+1}-1}(A_{lt_1}A_{t_1t_2}\cdots A_{t_ki_{s+1}})\right]A_{i_{s+1}i_{s+2}}\cdots A_{i_kj}$$
	$$=\sum_{t_1\leq t_2\leq \cdots \leq t_k=l+1}^{i_{s+1}-1}(A_{ii}^{s-1} A_{il}A_{lt_1}A_{t_1t_2}\cdots A_{t_ki_{s+1}}A_{i_{s+1}i_{s+2}}\cdots A_{i_kj}).$$
		Since $A$ is a sign $k$-potent, $A_{il}A_{lt_1}A_{t_1t_2}\cdots A_{t_{k-1}t_k}$ is a subpattern of $A_{it_k}$. 
	
    So, $A_{ii}^{s-1} A_{il}A_{lt_1}A_{t_1t_2}\cdots A_{t_ki_{s+1}}A_{i_{s+1}i_{s+2}}\cdots A_{i_kj}$ is a subpattern of $A_{ii}^{s-1} A_{it_k}A_{t_ki_{s+1}}\cdots A_{i_kj}$. Since $t_k\geq l+1$, by the definition of $l$, $A_{ii}^{s-1} A_{it_k}A_{t_ki_{s+1}}\cdots A_{i_kj}=0$. Therefore, 
	$$A_{ii}^{s-1} A_{il}A_{li_{s+1}}\cdots A_{i_kj}=0,$$ which is a contradiction.
\end{proof}

\begin{theorem} \label{2.4}
	Let $A$ be a reducible sign $k$-potent sign pattern in cyclic normal form. Then $A$ allows $k$-potence if and only if $A$ is in \textbf{PPO}.
\end{theorem}
\begin{proof}
	First, suppose that $A$ is a sign $k$-potent sign pattern that allows $k$-potence, then $A$ is in \textbf{PPO}  by Theorems \ref{5.1aa} and \ref{5.2}. 
	
	Conversely, suppose that $A$ is a sign $k$-potent sign pattern in cyclic normal form, which is in \textbf{PPO}. Now by Lemma \ref{11}, if we consider the subblock partition of $A$ induced by the cyclic normal form of each irreducible diagonal block of the Frobenius normal form of $A$, then the $(p,q)$-th subblock of $ A_{ij}$ where $i< j$ is of the form $(A_{ij})_{pq}=\alpha_{pq}J_{n_p\times n_q}$ for $p=1,2,...,m_i$; $q=1,2,...,m_j$ where $m_i$, $m_j$ is the number of diagonal blocks of $A_{ii}$, $A_{jj}$ respectively, and $\alpha_{pq}\in \{+,-,0\}$, and $J_{n_p\times n_q}$ is the sign pattern whose all entries are pluses.
    
	Define a matrix $B=[ B_{ij}]\in \mathcal{Q}(A)$ such that if $A_{ii}$ is nonzero, then $B_{ii}$ is defined as in (\ref{1.4aa}).
	Let $A_{ij}$, $i<j$, be nonzero, and let at least one of $A_{ii}$, $A_{jj}$ be nonzero.
	Without loss of generality, assume that $A_{ii}$ is a nonzero block. Since $A$ is in \textbf{PPO}, therefore $A_{jj}$ is a zero block and $A_{ij}$ is a column of $m_i$ blocks. The $p$-th subblock of $B_{ij}$ is defined as
	$$(B_{ij})_{p}=\left[ \frac{\beta_{p}}{n_p}\right] ~ whenever~ (A_{ij})_{p}=\alpha_{p}J_{n_p\times 1},~ p=1,2,...,m_i$$$$and~\beta_{p}\in \{0,1,-1\}~with~ \sgn(\beta_{p})=\alpha_{p}.$$
	If $A_{ij}$ is nonzero and both $A_{ii}$, $A_{jj}$ are zero blocks, then define $B_{ij}$ as 
	$$B_{ij}=\sum_{t_1\leq t_2\leq \cdots \leq t_k=i+1}^{j-1}(B_{it_1}B_{t_1t_2}\cdots B_{t_kj}).$$
	
	Therefore, in both cases $B_{ii}^{k+1}=B_{ii}$ for all $i$ and $(B^{k+1} )_{ij}=B_{ij}$ for all $i,j$; $i<j$.
	So, $B^{k+1}=B$ and $A$ allow $k$-potence.
\end{proof}

\begin{theorem} \label{2.5}
	Let $A$ be a reducible sign $k$-potent sign pattern in cyclic normal form. Then $A$ allows $k$-potence if and only if $\red(A)$ allows $k$-potence.
\end{theorem}

\begin{proof}
	Note that $A$ is in \textbf{PPO} if and only if $\red(A) $ is in \textbf{PPO}. Hence, the result follows from Theorem \ref{2.4}.

\end{proof}
Note that the previous result was also obtained by J. Stuart (Theorem~12, \cite{2017}). However, here we give a different proof of the same result.

	%%%%%%%%%%%%%%%%%%%%%%%%%%%%%%%%%%%%%%%%%%%%%%%%%%%%%%%%%

			\section{Acknowledgement} The research work of Partha Rana was supported by the Council of Scientific and Industrial Research (CSIR), India (File Number 09/731(0186)/2021-EMR-I).
			
\end{document}